\documentclass[a4paper,12pt]{article}

\usepackage[a4paper, left=2.5cm, right=2.5cm, top=2.5cm, bottom=2cm]{geometry}

\usepackage{mathtools}
\usepackage{amsfonts, amsthm, amsmath, amssymb}
\usepackage{algorithm}
\usepackage{algorithmicx}
\usepackage{xcolor}
\usepackage{soul}
\usepackage{todonotes}
\usepackage[shortlabels]{enumitem}
\usepackage{comment}
\usepackage{algpseudocode}
\usepackage{lipsum}
\usepackage{booktabs}
\usepackage{hyperref}
\usepackage{subcaption}

\newtheorem{theorem}{Theorem}
\newtheorem{lemma}[theorem]{Lemma}

\newtheorem{corollary}[theorem]{Corollary}

\theoremstyle{definition}
\newtheorem{assumption}{Assumption}
\newtheorem{definition}[theorem]{Definition}
\newtheorem{remark}[theorem]{Remark}

\newcommand{\R}{\mathbb{R}}
\newcommand{\N}{\mathbb{N}}

\DeclareMathOperator{\dist}{dist}

\DeclareMathOperator{\prox}{prox}
\DeclareMathOperator*{\argmin}{argmin}
\DeclareMathOperator{\dom}{dom}

\title{\textbf{An Inexact Regularized Proximal Newton Method without Line Search}}
\author{Simeon vom Dahl\thanks{University of Würzburg, Institute of Mathematics, Emil-Fischer-Str. 30, 97074 Würzburg, Germany; simeon.vomdahl@uni-wuerzburg.de} \and Christian Kanzow\thanks{University of Würzburg, Institute of Mathematics, Emil-Fischer-Str. 30, 97074 Würzburg, Germany;
christian.kanzow@uni-wuerzburg.de}}
\date{\today}

\begin{document}
\maketitle

\textbf{Abstract.} In this paper, we introduce an inexact regularized proximal Newton method (IRPNM) that does not require any line search. The method is designed to minimize the sum of a twice continuously differentiable function $f$ and a convex (possibly non-smooth and extended-valued) function $\varphi$. Instead of controlling a step size by a line search procedure, we update the regularization parameter in a suitable way, based on the success of the previous iteration. The global convergence of the sequence of iterations and its superlinear convergence rate under a local Hölderian error bound assumption are shown. Notably, these convergence results are obtained without requiring a global Lipschitz property for $ \nabla f $, which, to the best of the authors' knowledge, is a novel contribution for proximal Newton methods. To highlight the efficiency of our approach, we provide
numerical comparisons with an IRPNM using a line search globalization 
and a modern FISTA-type method. \\




\section{Introduction}

We are interested in solving the composite optimization problem
\begin{equation}
\min_{x \in \R^n} F(x):=f(x)+\varphi(x) \quad \text{with } f(x) := \psi(Ax-b),
\label{problem}
\end{equation}
where $A \in \R^{m \times n}$ and $b \in \R^m$ represent some given
data, and with
$\psi \colon \R^m \to \overline \R := \R \cup \{ \infty\}$ and $\varphi \colon \R^n \to \overline \R$ being proper lower semicontinuous (lsc) functions 
satisfying the following conditions.

\begin{assumption}
\begin{enumerate}[label=(\alph*), ref=\ref{assumption_general}(\alph*)]
\item $\psi$ is twice continuously differentiable on an open set containing $A(\Omega)-b$, where $\Omega \supseteq \dom\varphi$ is a closed subset of $\R^n$, \label{ass_psi}
\item $\varphi$ is convex and continuous on its domain $\dom \varphi$, \label{ass_phi}
\item $F$ is bounded from below, i.e., $F^*:=\inf_{x \in \R^n} F(x) > -\infty$. \label{ass_inf_F}
\end{enumerate}
\label{assumption_general}
\end{assumption}

\noindent
From this structure, it is clear that the objective function $F \colon \R^n \to \overline{\R}$ is also proper and lower semicontinuous, but possibly nonsmooth and nonconvex. Assumption \ref{ass_psi} and the chain rule guarantee that $f$ is twice continuously differentiable on an open set containing $\Omega$ with 
\begin{equation}
\label{eq_grad_hess_f}
\nabla f(x) = A^\top \nabla \psi(Ax-b), \quad \nabla^2 f(x) = A^\top \nabla^2 \psi(Ax-b) A \quad \text{for all } x \in \Omega.
\end{equation}
Note that model \eqref{problem} along with the above assumptions is almost the same as in \cite{liu_pan_wu_yang_2024}. The only difference lies in assumption \ref{ass_inf_F}, where \cite{liu_pan_wu_yang_2024} requires coerciveness of $F$ instead of boundedness from below. Note that this coercivity is a much stronger condition. In particular, together with the assumed lower
semicontinuity assumption, it implies that all sublevel sets are compact,
so that \eqref{problem} has a compact set of minimizers. Moreover,
it guarantees global Lipschitz continuity of the gradient $\nabla f$ on all sublevel sets of $F$. The elimination of the coercivity requirement on $F$ is therefore significant. 

Problems of type \eqref{problem} frequently arise in various fields, including statistics, machine learning, image processing, and many others. Notably, the well-known LASSO problem, as introduced by Tibshirani in \cite{tibshirani_1996}, represents a special (convex) instance of \eqref{problem}.
Applications to compressive sensing problems are discussed in detail
in \cite{FoR-13}. Machine learning applications like low rank approximations
are extensively treated in the book \cite{Mar-19}, and dictionary
learning algorithms are surveyed in the monograph \cite{DuI-18}.
Matrix completion problems, both convex and nonconvex, have been extensively explored in the past \cite{marjanovic_solo_2012, yu_peng_yue_2019}. Additionally, \cite{bian_chen_2015} serves as a representative example of the numerous applications of \eqref{problem} in the field of image processing.

\subsection{Related Work}

Proximal methods have a long history, beginning with Martinet's proximal point algorithm \cite{martinet_1970, martinet_1972}. Later, Rockafellar generalized the theory and applied it to convex minimization problems \cite{rockafellar_1976, rockafellar_1976b}. The first proximal method for nonconvex problems of the form \eqref{problem} was the proximal gradient method introduced by Fukushima and Mine \cite{fukushima_mine_1981}. Subsequently, several proximal gradient methods emerged, including the well-known Iterative Shrinkage/Thresholding Algorithm (ISTA) and its accelerated version, FISTA, introduced by Beck and Teboulle \cite{beck_teboulle_2009}. New FISTA-type methods continue to be introduced, such as the recent example in \cite{liang_monteiro_2023} by Liang and Monteiro.


The idea of proximal Newton methods is to find in each step, for a current iterate $x^k$, an approximate minimizer $y^k$ of the subproblem
\begin{equation}
\min_x \hat q_k(x):=f(x^k)+\nabla f(x^k)^\top (x-x^k)+\frac{1}{2}(x-x^k)^\top G_k(x-x^k)+\varphi(x),
\label{subproblem}
\end{equation}
where $G_k$ is either the Hessian $\nabla^2 f(x^k)$ or a suitable approximation of the exact Hessian. The main difference to proximal gradient methods is the incorporation of second-order information, which leads to a faster convergence rate due to a better local approximation of the nonlinear function $ f $. On the other hand, iterative methods for the solution of the subproblem \eqref{subproblem} usually take longer due to the more complex nature of this
subproblem. In fact, note that the 
proximal Newton method reduces to the proximal gradient method if $G_k$ is a multiple of the identity matrix at each iteration, so that the proximal
gradient subproblem is (usually) easier to solve, in several applications
even analytically.

Stationary points of \eqref{problem} are given by the solutions of the generalized equation $0 \in \nabla f(x) + \partial \varphi (x)$, where $\partial \varphi (x)$ denotes the (convex) subdifferential of $\varphi$ at $x$, and this inclusion can be rewritten as 
\begin{equation*}
	r(x)= 0
\end{equation*}
for a certain residual function, see \eqref{Eq:Residual-function} below
for the precise definition. Similary, the stationary conditions 
of the subproblems \eqref{subproblem} reduce to the solution of the partially linearized generalized equation at iterate $x^k$:
\begin{equation}
0 \in \nabla f(x^k) + G_k(x-x^k) + \partial \varphi (x).
\label{eq_par_lin_gen_eq}
\end{equation}
Various results on the convergence of iterative methods for solving \eqref{eq_par_lin_gen_eq} can be found in the literature. Fischer \cite{fischer_2002} proposes a very general iterative framework for solving generalized equations and proves local superlinear and quadratic convergence of the resulting iterates under an upper Lipschitz continuity assumption of the solution set map of a perturbed generalized equation. Early proximal Newton methods were designed for special instances of \eqref{problem}, mostly with convex $\psi$ and $\varphi$ such as GLMNET \cite{friedman_hoefling_tibshirani_2007,friedman_hastie_tibshirani_2010} and newGLMNET \cite{yuan_ho_lin_2012} for generalized linear models with elastic-net penalties, QUIC \cite{hsieh_dhillon_ravikumar_sustik_2011} for the $l_1$-regularized Gaussian maximum likelihood estimator and the Newton-Lasso method \cite{oztoprak_nocedal_rennie_olsen_2012} for the sparse inverse covariance estimation problem. 

Lee et al. \cite{lee_yuekai_saunders_2014} were the first to propose a generic version of the exact proximal Newton method for \eqref{problem} with convex $f$. They assume that $\nabla f$ is Lipschitz continuous and show global convergence under the uniform positive definiteness of $\{G_k\}$ and local quadratic convergence under the strong convexity of $f$ and the Lipschitz continuity of $\nabla^2 f$. Byrd et al. \cite{byrd_nocedal_oztoprak_2016}, considering \eqref{problem} with the $l_1$-regularizer $\varphi(x) = \lambda \Vert x \Vert_1$, propose an implementable inexactness criterion for minimizing $\hat q_k$ while achieving global convergence, and local fast convergence results under similar assumptions to \cite{lee_yuekai_saunders_2014}. Their global convergence theory also works for nonconvex $f$. Yue et al.\ \cite{yue_zhou_so_2019} used the inexactness criterion and the line search procedure of \cite{byrd_nocedal_oztoprak_2016} to develop an inexact proximal Newton method with a regularized Hessian and proved its local superlinear and quadratic convergence under the Luo-Tseng error bound condition \cite{luo_tseng_1992}, which is significantly weaker than the strong convexity assumption on $f$. Mordukhovich et al.\ \cite{mordukhovich_2023} further improve on \cite{yue_zhou_so_2019} by eliminating an impractical assumption where the parameters of their method satisfy a condition involving a constant that is difficult to estimate. They also prove local superlinear convergence under the metric $q$-subregularity of $\partial F$ for $q \in \left( \frac{1}{2},1 \right)$, a condition even weaker than the Luo-Tseng error bound. Their entire analysis, however, concentrates
on convex functions $ f $.

While proximal Newton-type methods for problem \eqref{problem} with convex $f$ have been extensively explored in the past, there has been limited research to date on the case where $f$ is nonconvex. In the previously referenced paper \cite{byrd_nocedal_oztoprak_2016}, global convergence was established with nonconvex $f$ and the $l_1$-regularizer, albeit still requiring a strong convexity assumption on 
$f$ for the local convergence theory. Lee and Wright \cite{lee_wright_2019} investigated an inexact proximal Newton method, presenting a sublinear global convergence rate result on the first-order optimality condition for general choices of $G_k$, with the sole assumption of $\nabla f$ being Lipschitz continuous. Combining the advantages of proximal Newton and proximal gradient methods, Kanzow and Lechner \cite{kanzow_lechner_2021} introduced a globalized inexact proximal Newton method (GIPN). In this approach, a proximal gradient step is taken whenever the proximal Newton step fails to satisfy a specified sufficient decrease condition. They proved global convergence with a local superlinear convergence rate under the local strong convexity of $F$ and uniformly bounded positive definiteness of $G_k$.
Inspired by the work \cite{ueda_yamashita_2010} for smooth nonconvex optimization problems, Liu et al.\ \cite{liu_pan_wu_yang_2024} extended the theory of \cite{mordukhovich_2023} to the case of \eqref{problem}, where $f$ is allowed to be nonconvex. Instead of the metric $q$-subregularity on $\partial F$, they assumed that accumulation points of the iterate sequence satisfy a Hölderian local error bound condition on the set of so-called strongly stationary points  to show convergence of the iterates with a local superlinear convergence rate. They achieve a local superlinear convergence rate without $F$ being locally strongly convex. However, they require that $F$ is level-bounded. 

All aforementioned works employed a proximal Newton-type method in conjunction with an appropriate line search strategy for global convergence. There has been minimal exploration of proximal Newton methods with alternative globalization strategies. Yamashita and Ueda \cite{ueda_yamashita_2014} investigated regularized Newton methods for smooth unconstrained problems, achieving global convergence by adjusting the regularization parameter based on the success of the previous iteration, similar to a trust-region scheme. As of the authors' knowledge, the method described in the PhD thesis \cite[Chapter 4]{lechner_2022} remains the only instance where this globalization strategy was applied within the framework of proximal Newton-type methods.

Historically, the global Lipschitz continuity of $\nabla f$ has been a standard assumption for the convergence analysis of proximal gradient and proximal Newton methods. While recent works have successfully eliminated this assumption for proximal gradient methods (see, for example, \cite{BST-14,kanzow_mehlitz_2022,jia_kanzow_mehlitz_2023,Marchi-23}), there are no known comparable results for proximal Newton methods.

\subsection{Our Contributions}

In this work, we present a proximal Newton method without a line search for problem \eqref{problem} under assumption \ref{assumption_general}. Building upon the selection in \cite{liu_pan_wu_yang_2024}, we employ the following expression as the regularized Hessian at iteration $x^k$:
\begin{equation}\label{Eq:Gk}
G_k = \nabla^2 f(x^k) + \Lambda_k A^\top A + \nu_k {\overline r_k}^\delta I
\end{equation}
with
\begin{equation}\label{Eq:Lambdak}
  \Lambda_k := a \left[ -\lambda_{min}\left( \nabla^2 \psi \left(Ax^k-b\right)\right)\right]_+, \quad a \geq 1, \quad
  \text{and} \quad \delta \in (0,1]. 
\end{equation}
Recall from \eqref{eq_grad_hess_f} that $ G_k $ can be rewritten as
\begin{equation*}
	G_k = A^\top \big( \nabla^2 \psi (Ax^k -b) + \Lambda_k I \big) A +
	\nu_k {\overline r_k}^\delta I,
\end{equation*}
hence the definition of $ \Lambda_k $ immediately implies that the matrix
$ G_k $ is positive definite (the first term is positive semidefinite). 
The only difference
to \cite{liu_pan_wu_yang_2024} resides in the final term, where the sequence $\{\overline r_k\}_{k \in \N_0}$ is recursively given by 
\begin{equation}
\overline r_0 := \Vert r(x^0) \Vert \text{ and } \overline r_{k+1} = \begin{cases}
\Vert r(\hat x^k) \Vert, & \text{if } \Vert r(\hat x^k) \Vert \leq \eta \overline r_k \\
\overline r_k, & \text{otherwise}
\end{cases}
\text{ for } k \in \N_0,
\end{equation}
with $\hat x^k$ being an approximate solution of subproblem \eqref{subproblem}, $\eta \in (0,1)$, and $ r $ is the residual function already mentioned before and
formally defined in \eqref{Eq:Residual-function} below. Additionally, the regularization parameter $\nu_k$ follows an update strategy akin to \cite{ueda_yamashita_2014} and \cite{lechner_2022}, detailed in Section~\ref{Sec:Algorithm}. Notably, the sequence $\{G_k\}$ is not uniformly  positive definite, since $\{\overline r_k\}$ converges to $0$, as clarified later.

We establish the global convergence of the iterate sequence and its convergence rate of $q(1+\delta)>1$, assuming the existence of an accumulation point that satisfies a local Hölderian error bound of order $q > \max\left\{\frac{1}{1+\delta}, \delta \right\}$ on the set of strongly stationary points. In comparison with \cite{liu_pan_wu_yang_2024}, our approach reproduces essentially the same convergence results, employing an update strategy for the regularization parameter instead of a line-search technique. Most notably, we eliminate the requirement for $F$ to be coercive. The coerciveness guarantees a compact minimizer set for \eqref{problem}, as well as the Lipschitz continuity of $\nabla f$ on all sublevel sets of $F$. It is noteworthy that this global Lipschitz continuity of the gradient of $f$ has been a standard assumption in order to prove convergence of the iterate sequence. To the best of the authors' knowledge, this work is the first to eliminate this assumption. 

Utilizing the dual semismooth Newton augmented Lagrangian method (SNALM) developed in \cite{liu_pan_wu_yang_2024} as a subproblem solver, we compare the performance of our method (IRPNM-reg) with the line search based inexact regularized proximal Newton method (IRPNM-ls) from \cite{liu_pan_wu_yang_2024} and AC-FISTA \cite{liang_monteiro_2023} on five distinct test problems. 

\subsection{Notation}

In this paper, $\N = \{1,2,3,...\}$ denotes the set of positive integers and we write $\N_0 := \N \cup \{0\}$. The extended real numbers are given by $\overline \R := \R \cup \{+\infty\}$. For $a \in \R$ we write $a_+ := \max(0,a)$. For $x \in \R^n$, $\Vert x \Vert$ represents the Euclidean norm, $B_\varepsilon(x)$ stands for the closed ball around $x$ with radius $\varepsilon > 0$, and $\dist(x,C)$ denotes the Euclidean distance from $x$ to a closed set $C \subseteq \R^n$. The set $\mathbb{S}^n$ comprises all real symmetric matrices of dimension $n \times n$ and $\mathbb{S}_{++}^n$ is the set of all positive definite matrices in $\mathbb{S}^n$. For $M \in \mathbb{S}^n$, its spectral norm is denoted by $\Vert M \Vert$ and $M \succeq 0$ indicates that $M$ is positive semidefinite. The smallest eigenvalue of $M$ is denoted by $\lambda_{min}(M)$. The identity matrix is denoted by $I$, with its dimension being evident from the context. The domain of a function $g \colon \R^n \to \overline \R$ is defined as $\dom g := \left\{x \in \R^n \mid g(x) < \infty \right\}$ and $g$ is called proper if $\dom g \neq \emptyset$.

\section{Preliminaries}\label{Sec:Prelims}

This section summarizes some background material from variational 
analysis that will be important in our subsequent sections.

First of all, we denote by $\partial F(x)$ the \emph{basic (or limiting or
Mordukhovich) subdifferential} of $F$ at $x$, see the standard
references \cite{rock_wets_1998,Mor-18} for more details. Its precise definition plays no
role in our subsequent discussion since only some of its basic properties
will be used. In particular, it is known that, for convex functions, this basic subdifferential simplifies to the well-known convex subdifferential. 
Furthermore, according to \cite[Exercise 8.8(c)]{rock_wets_1998}, for any $x \in \dom\varphi$, it holds that $\partial F(x) = \nabla f(x) + \partial \varphi(x)$. 

Based on this notion, we now introduce two stationarity concepts, the first
one being the standard stationarity condition for composite optimization
problems, the second one being a stronger concept taken from  \cite{liu_pan_wu_yang_2024}.

\begin{definition}\label{Def:StatConditions}
A point $x \in \dom \varphi$ is called a
\begin{itemize}
	\item[(a)] \emph{stationary point} of problem \eqref{problem} if $0 \in \partial F(x) \ \big( = \nabla f(x) + \partial \varphi(x) \big) $;
	\item[(b)] \emph{strongly stationary point} of problem \eqref{problem} if
	it is a stationary point which, in addition, satisfies $\nabla^2\psi(Ax-b) \succeq 0$.
\end{itemize}
We denote by $S^*$ and $X^*$ the sets of all stationary and strongly stationary
points, respectively.
\end{definition}

\noindent
Note that Assumption~\ref{ass_psi} together with the outer semicontinuity of $\partial \varphi$ implies that the $\mathcal S^*$ and $\mathcal X^*$ are
closed. In contrast to the situation discussed in \cite{liu_pan_wu_yang_2024},
we stress that both $\mathcal S^*$ and $\mathcal X^*$ might be empty in our
setting due to the removal of the coerciveness assumption on $F$. We
further note that there might be stationary points (i.e., $\mathcal S^* \neq \emptyset$), while $\mathcal X^*$ is still empty. A local minimizer is 
always a stationary point, but is not guaranteed to be strongly stationary, and the converse may not be true either. 

Proximal Newton-type methods rely on the proximity operator. For a proper, lower semicontinuous and convex function $g \colon \R^n \to \overline{\R}$, the proximity operator $\prox_g \colon \R^n \to \R^n$ is defined by
\[
\prox_g(x):=\argmin_y\left\{g(y)+\frac{1}{2} \Vert y-x \Vert^2\right\}.
\]
The objective function $g(y)+\frac{1}{2}\Vert y - x \Vert^2$ is strongly convex on $\dom g$. This ensures a unique minimizer for every $x \in \R^n$, i.e., the proximity operator is well-defined. Moreover, the operator is nonexpansive, signifying Lipschitz continuity with constant one. It also satisfies the crucial relationship
\begin{equation}
	\label{eq_prox_subdiff}
	y = \prox_g(x) \Longleftrightarrow y \in x-\partial g(y),
\end{equation}
which shows that
\begin{equation}
	\label{eq_stat_point_char}
	x \in \mathcal S^* \Longleftrightarrow -\nabla f(x) \in \partial \varphi(x) \Longleftrightarrow x = \prox_\varphi(x-\nabla f(x)).
\end{equation}
Motivated by this, the \emph{residual or prox-gradient mapping} is defined by
\begin{equation}\label{Eq:Residual-function}
	r(x) := x - \prox_\varphi(x-\nabla f(x)), \quad x \in \R^n.
\end{equation}
Consequently, $x \in \R^n$ is a stationary point of $F$ if and only if $r(x)=0$. Hence, the norm of $r(x)$ can be used to measure the stationarity of $x$.

\section{The Algorithm and its Basic Properties}\label{Sec:Algorithm}

Consider a fixed iteration $k \geq 0$ with a current iterate $x^k \in \R^n$.
Then the core task of proximal Newton methods lies in solving the subproblem
\begin{equation}
\min_x q_k(x):=f(x^k)+\nabla f(x^k)^\top (x-x^k)+\frac{1}{2}(x-x^k)^\top \nabla^2f(x^k)(x-x^k)+\varphi(x).
\label{proximal_newton_cvx}
\end{equation}
The first part of $q_k$ provides a quadratic approximation of the smooth function $f$. However, since $f$ is not necessarily convex, $\nabla^2f(x^k)$ 
may not be positive semidefinite and, hence, $q_k$ may not be convex. To address this difficulty, we consider the matrix
\[
H_k := \nabla^2f(x^k) + \Lambda_kA^\top A
\]
with $ \Lambda_k $ defined in \eqref{Eq:Lambdak}. Recall from the discussion
following \eqref{Eq:Lambdak} that $ H_k $, simply by definition of
$ \Lambda_k $, is positive semidefinite. Furthermore, given some 
regularization parameter $ \mu_k > 0 $, the corresponding matrix $ G_k $
from \eqref{Eq:Gk}, which is given by
\begin{equation*}
	G_k = H_k + \mu_k I,
\end{equation*}
is then automatically positive definite. This implies that the resulting
subproblem
\begin{equation}
\min_x\hat q_k(x):=f(x^k)+\nabla f(x^k)^\top (x-x^k)+\frac{1}{2}(x-x^k)^\top G_k(x-x^k)+\varphi(x),
\label{reg_proximal_newton_cvx}
\end{equation} 
has a strongly convex objective function and, thus, a unique solution.
Throughout this paper, we write
\begin{equation*}
	\overline x^k := \mathrm{argmin}_x \hat q_k(x)
\end{equation*}
for this unique minimum. From a numerical point of view, computing
this minimum exactly might be very demanding, and we therefore require
an inexact solution only. We denote this inexact solution by $\hat x^k$.
In order to prove suitable global and local convergence results, this
inexact solution has to satisfy certain criteria which measure the 
quality of the inexact solution. Here, we assume that the inexact
solution $ \hat x^k $ is computed in such a way that the conditions
\begin{align}
\label{eq_inexactness}
\Vert R_k(\hat x^k) \Vert \leq \theta \min\left\{ \Vert r(x^k) \Vert, \Vert r(x^k) \Vert^{1+\tau} \right\} \text{ and } F(x^k)- \hat q_k(\hat x^k) \geq \frac{\alpha\mu_k}{2}\Vert \hat x^k - x^k \Vert^2
\end{align}
hold, where $\alpha, \theta \in (0,1)$ and $\tau \geq \delta$ are 
certain constants, and the residual $R_k$ is defined by
\begin{align*}
R_k(x) := x-\prox_{\varphi}\left(x-\nabla f(x^k) - (H_k+\mu_kI)(x-x^k)\right).
\end{align*}
Note that $ R_k $ is the counterpart of the residual $ r $ from 
\eqref{Eq:Residual-function} for the subproblem 
\eqref{reg_proximal_newton_cvx}. In particular, and similar to
\eqref{eq_stat_point_char}, a vector $x$ is an optimal solution of \eqref{reg_proximal_newton_cvx} if and only if $ R_k(x)  = 0$. This 
explains why the first condition from \eqref{eq_inexactness} serves as
an inexactness criterion. Regarding the second condition, we refer to 
Lemma~\ref{lemma_inexactness} below for a justification.

Typically, see the recent papers \cite{mordukhovich_2023} and \cite{liu_pan_wu_yang_2024}, these (regularized) proximal Newton-type methods are combined with an appropriate line search strategy to achieve global convergence. In this work, our objective is to attain global convergence by controlling the regularization parameter itself, depending on the success of the previous iteration. This idea has already been used in \cite{ueda_yamashita_2014} with a regularized Newton method for the minimization of a twice differentiable function. Recently, in the PhD thesis \cite{lechner_2022}, it has been established for proximal Newton methods in the composite setting. To assess the success of a candidate $\hat x^k$, we consider the ratio 
\begin{equation}
\rho_k:=\frac{\text{ared}_k}{\text{pred}_k}
\end{equation}
between the actual reduction
\begin{equation}
\text{ared}_k:=F(x^k)-F(\hat x^k)
\label{reg_pred_red_cvx}
\end{equation}
and the predicted reduction
\begin{equation}
\text{pred}_k:=F(x^k)-q_k(\hat x^k).
\label{reg_actual_red_cvx}
\end{equation}
It is important to note that for the predicted reduction, we use the unregularized approximation $q_k$ instead of $\hat q_k$. From the second condition in \eqref{eq_inexactness} it follows that
\begin{align}
\label{eq_pred}
\begin{split}
\text{pred}_k &= F(x^k) - q_k(\hat x^k)  = F(x^k) - \hat q_k(\hat x^k) + \frac{1}{2} (\hat x^k-x^k)^\top(\Lambda_k A^\top A+\mu_kI)(\hat x^k-x^k) \\
&\geq F(x^k) - \hat q_k(\hat x^k) + \frac{\mu_k}{2}\Vert \hat x^k - x^k \Vert^2 \geq \frac{\mu_k}{2}\Vert \hat x^k-x^k \Vert^2
\end{split}
\end{align}
for all $k \geq 0$. In particular the predicted reduction is positive if $x^k$ is not already a stationary point of \eqref{problem}.  This follows from the
following simple observation.

\begin{remark}\label{Rem:Positive}
If $ x^k = \hat x^k $, then $ x^k $ is already a stationary point
of \eqref{problem}. Hence, $ \text{pred}_k > 0 $ at all iterations $ k $
such that $ x^k $ is not already a stationary point.
\end{remark}

\begin{proof}
Let  $ x^k = \hat x^k $. Then the definitions of the corresponding 
residual functions yield $ R_k (\hat x^k) = R_k (x^k) = r(x^k) $. 
Since $ \theta \in (0,1) $, we then obtain $ r(x^k) = 0 $ from 
first inexactness test in \eqref{eq_inexactness}.
\end{proof}

\noindent
We are now ready to present our algorithm.

\begin{algorithm}[H]
\caption{Regularized proximal Newton method} 
\label{algo_proxRegNewton_cvx}
\begin{algorithmic}[1]
\State Choose $x^0 \in \dom\varphi$ and parameters $c_1 \in (0,1); \, c_2 \in (c_1,1); \, \sigma_1 \in (0,1); \, \sigma_2>1; \, \eta \in (0,1);\, \theta \in (0,1); \, \alpha \in (0,1); \, a \geq 1; \, 0 < \nu_{min} \leq \nu_0 \leq \overline\nu ; \, 0 < \delta \leq 1;\, \tau \geq \delta;\, p_{min} \in (0,1/2);\, \kappa > 1+\delta. $ Set $ k:= 0; \ \overline r_0 := \Vert r(x^0) \Vert; \, \mu_0 := \nu_0 \overline r_0^\delta.$
\For{$k=0,1,2,...$}
\State Compute an inexact solution $\hat x^k$ of the proximal regularized Newton subproblem \eqref{reg_proximal_newton_cvx} satisfying the inexactness criterion \eqref{eq_inexactness}. 
\State Set $d^k:=\hat x^k-x^k$.
\State Compute $\text{pred}_k$, $\text{ared}_k$ and $\rho_k$.
\If{$\text{pred}_k \leq p_{min}(1-\theta)\Vert d^k \Vert \min\{\Vert r(x^k) \Vert, \Vert r(x^k) \Vert^\kappa\}$ OR $\rho_k \leq c_1$}  
\State Set $x^{k+1}=x^k, \nu_{k+1}=\sigma_2\nu_k$.\Comment{unsuccessful iteration}
\Else
\State Set $x^{k+1} = \hat x^k$.
\If{$\rho_k \leq c_2$}
\State Set $\nu_{k+1} = \min\{\nu_k,\overline\nu\}$. \Comment{successful iteration} \label{algo_line_K_commands_1} 
\Else
\State Set $\nu_{k+1} = \min\{\max\{\sigma_1\nu_k,\nu_{min}\},\overline \nu \}$. \Comment{highly successful iteration} \label{algo_line_K_commands_2} 
\EndIf
\EndIf
\If{$\Vert r(x^{k+1}) \Vert \leq \eta\overline r_k$} \Comment{$k+1 \in \mathcal K$} \label{algo_line_K}
\State $\overline r_{k+1} = \Vert r(x^{k+1}) \Vert$. \label{algo_line_K_command}
\Else
\State $\overline r_{k+1} = \overline r_k$.
\EndIf
\State $\mu_{k+1}=\nu_{k+1}\overline r_{k+1}^\delta$.
\EndFor
\end{algorithmic}
\end{algorithm}

\noindent
The basic idea of Algorithm~\ref{algo_proxRegNewton_cvx} is to solve,
iteratively, the proximal regularized Newton 
subproblem \eqref{reg_proximal_newton_cvx}
and either to accept the inexact solution as the new iterate, 
provided that this makes a sufficient progress in the sense of the
tests in line 6, or to stay at the current point and enlarge the  regularization parameter. The steps between lines 10 and 20 are devoted to a very careful update of the parameter $ \nu_k $ as well as $\overline r_k$, hence
of the regularization parameter $ \mu_k $ in line 21, since this update
is essential especially for the local convergence analysis where we
prove fast local convergence under fairly mild assumptions.

In the remaining part, we state a number of basic properties which might,
partially, explain some of these careful updates.

\begin{lemma}
\label{lemma_r_dec}
\begin{enumerate}[(a)]
\item The sequence $\{\overline r_k\}$ is monotonically decreasing, \label{eq_lemma_r_dec_1}
\item For all $k \geq 0$ it holds that $\Vert r(x^k) \Vert > \eta \overline r_k$. \label{eq_lemma_r_dec_2}
\item For all $k \geq 0$ it holds that $\overline r_k \geq \min\{\Vert r(x^j) \Vert \mid 0 \leq j \leq k \}$.
\label{eq_lemma_r_dec_3}
\end{enumerate}
\end{lemma}

\begin{proof}
(a) We consider an iteration $k \geq 0$. If the condition $\Vert r(x^{k+1}) \Vert \leq \eta \overline r_k$ in line \ref{algo_line_K} of Algorithm \ref{algo_proxRegNewton_cvx} is satisfied, we get $\overline r_{k+1} = \Vert r(x^{k+1}) \Vert \leq \eta \overline r_k < \overline r_k$. Otherwise, the algorithm directly sets $\overline r_{k+1} = \overline r_k$. Combining these two cases shows that $\overline r_{k+1} \leq \overline r_k$. Hence, the sequence $\{\overline r_k\}$ is monotonically decreasing. \medskip

\noindent
(b) For $k=0$ this property obviously holds. Suppose now that this property holds for some $k \in \N_0$. If the condition in line \ref{algo_line_K} is satisfied at iteration $k$, the algorithm directly sets $\overline r_{k+1} = \Vert r(x^{k+1}) \Vert$. If $k$ is an unsuccessful iteration, then we get $\Vert r(x^{k+1}) \Vert = \Vert r(x^k) \Vert > \eta \overline r_k \geq \eta \overline r_{k+1}$ by using the induction hypothesis together with $x^{k+1}=x^k$ and \ref{eq_lemma_r_dec_1}. In the remaining case it holds that $\Vert r(x^{k+1}) \Vert > \eta \overline r_k$ and iteration $k$ is successful or highly successful. Then we get
\[
\Vert r(x^{k+1}) \Vert > \eta \overline r_k = \eta \overline r_{k+1}.
\]
The combination of the three cases yields the result. \medskip

\noindent
(c) For $k=0$ this property obviously holds. Suppose now that this property holds for some $k \in \N_0$. If the condition in line \ref{algo_line_K} is satisfied at iteration $k$ we get 
\[
\overline r_{k+1} = \Vert r(x^{k+1}) \Vert \geq \min\{\Vert r(x^j) \Vert \mid 0 \leq j \leq k+1\}.
\]
Otherwise it holds that
\[
\overline r_{k+1} = \overline r_k \geq \min\{\Vert r(x^j) \Vert \mid 0 \leq j \leq k \} \geq \min\{\Vert r(x^j) \Vert \mid 0 \leq j \leq k+1 \},
\]
where we used the induction hypothesis in the first inequality.   
\end{proof}

\noindent
The following result contains some estimates regarding the sequence
$ \{ \nu_k \} $ and the corresponding sequence $ \{ \mu_k \} $ of
regularization parameters.

\begin{lemma}
\label{lemma_alg_prop_mu}
For an iteration $k \geq 0$ it holds that
\begin{enumerate}[(a)]
\item $\nu_k \geq \nu_{min}$,
\item $x^{k+1} = x^k$, $\nu_{k+1}=\sigma_2 \nu_k > \nu_k$ and $\mu_{k+1} = \sigma_2\mu_k > \mu_k$, if $k$ is unsuccessful, \label{eq_lemma_alg_prop_mu_unsucc}
\item $x^{k+1} = \hat x^k$, $\nu_{k+1} \leq \nu_k$ and $\mu_{k+1} \leq \mu_k$, if $k$ is successful or highly
successful.
\label{eq_lemma_alg_prop_mu_succ}
\end{enumerate}
\end{lemma}

\begin{proof}
(a) This statement follows recursively from $\nu_0 \geq \nu_{min}$ and the possible updates for $\nu_k$ in the algorithm. \medskip

\noindent
(b) If $k$ is an unsuccessful iteration, it follows by definition of the algorithm that $x^{k+1}=x^k$ and $\nu_{k+1} = \sigma_2 \nu_k$. From Lemma \ref{lemma_r_dec}\ref{eq_lemma_r_dec_2} it immediately follows that $\Vert r(x^{k+1}) \Vert = \Vert r(x^k) \Vert > \eta \overline r_k$, hence $\overline r_{k+1} = \overline r_k$ and eventually $\mu_{k+1} = \nu_{k+1}\overline r_{k+1}^\delta = \sigma_2 \nu_k \overline r_k^\delta = \sigma_2\mu_k > \mu_k$. \medskip

\noindent
(c) If $k$ is a successful or highly successful iteration, it follows by definition of the algorithm and statement (a)
that $x^{k+1} = \hat x^k$ and $\nu_{k+1} \leq \nu_k$. Using Lemma \ref{lemma_r_dec}\ref{eq_lemma_r_dec_1}, we then get $\mu_{k+1} = \nu_{k+1}\overline r_{k+1}^\delta \leq \nu_k\overline r_{k}^\delta = \mu_k$.
\end{proof}


\noindent
In the following we consider the set $\mathcal K \subset \N_0$ of iterations 
\[
\mathcal K := \{0\} \cup \left\{k \in \N \mid \, \text{The if-condition in line \ref{algo_line_K} was satisfied at iteration $k-1$}\right\}.
\]
Several properties for the iterates $ k $ belonging to this set are
summarized in the next result.

\begin{lemma}
\label{lemma_alg_prop_K}
For all iterations $k \in \mathcal K \setminus \{0\} \subset \N$, the following properties hold:
\begin{enumerate}[(a)]
    \item $\Vert r(x^k) \Vert \leq \eta \overline r_{k-1}$, \label{eq_lemma_alg_prop_K_dec}
    \item $\overline r_{k} = \Vert r(x^k) \Vert$, \label{eq_lemma_alg_prop_K_eq}
    \item iteration $k-1$ was successful or highly successful, \label{eq_lemma_alg_prop_K_succ}
    \item $\nu_k \leq \overline\nu $, \label{eq_lemma_alg_prop_K_max}
    \item $\mu_k \leq \overline\nu \Vert r(x^k) \Vert^\delta$. \label{eq_lemma_alg_prop_K_mu}
\end{enumerate}
\end{lemma}

\begin{proof}
Statements \ref{eq_lemma_alg_prop_K_dec} and \ref{eq_lemma_alg_prop_K_eq} follow directly from the if-condition in line \ref{algo_line_K} and the command in line \ref{algo_line_K_command}. If iteration $k-1$ was unsuccessful, then it would follow from Lemma \ref{lemma_r_dec}\ref{eq_lemma_r_dec_2} that $\Vert r(x^{k}) \Vert = \Vert r(x^{k-1}) \Vert > \eta \overline r_{k-1}$, a contradiction to $k \in \mathcal K$ according to \ref{eq_lemma_alg_prop_K_dec}. Hence \ref{eq_lemma_alg_prop_K_succ} holds. Assertion \ref{eq_lemma_alg_prop_K_max} then follows from \ref{eq_lemma_alg_prop_K_succ} and assertion \ref{eq_lemma_alg_prop_K_mu} follows from \ref{eq_lemma_alg_prop_K_eq} and \ref{eq_lemma_alg_prop_K_max}.
\end{proof}

\noindent 
The index set $ \mathcal K $ plays a central role in our convergence analysis.
The following result indicates why this set is so important.

\begin{lemma}
\label{lemma_equiv}
Let $\mathcal K = \{k_0,k_1,k_2,...\}$. For all $i \in \N_0$ it then holds that $\overline r_{k_{i+1}} \leq \eta \overline r_{k_i}$ and the following three statements are equivalent:
\begin{enumerate}[(i)]
\item $\mathcal K$ is an infinite set. 
\item $ \lim_{k \in \mathcal K} \Vert r(x^k) \Vert = 0$. 
\item $\liminf_{k \to \infty} \Vert r(x^k) \Vert = 0$.
\end{enumerate} 
\end{lemma}

\begin{proof}
Consider $i \in \N$ and $k_{i} \in \mathcal K$. From Lemma \ref{lemma_alg_prop_K}\ref{eq_lemma_alg_prop_K_dec}, \ref{lemma_alg_prop_K}\ref{eq_lemma_alg_prop_K_eq} and Lemma \ref{lemma_r_dec}\ref{eq_lemma_r_dec_1} it then follows immediately that $\overline r_{k_{i}} \leq \eta \overline r_{k_{i}-1} \leq \eta \overline r_{k_{i-1}}$. If $\mathcal K$ is an infinite set and using Lemma~\ref{lemma_r_dec}\ref{eq_lemma_r_dec_1}, this directly implies $\lim_{k \in \mathcal K} \overline r_k = \lim_{k \in \mathcal K} \Vert r(x^k) \Vert = 0$. From Lemma \ref{lemma_r_dec}\ref{eq_lemma_r_dec_3} it follows that $\liminf_{k \to \infty} \Vert r(x^k) \Vert = 0$. Suppose now that $\mathcal K$ is not an infinite set. Denote the last iteration in $\mathcal K$ by $\bar k$. Then it holds that $\overline r_k = \overline r_{\bar k}$ for all $k \geq \bar k$. It follows from Lemma \ref{lemma_r_dec}\ref{eq_lemma_r_dec_2} that $\Vert r(x^k) \Vert > \eta \overline r_k = \eta \overline r_{\bar k}$ for all $k \geq \bar k$. Hence, $\liminf_{k \to \infty} \Vert r(x^k) \Vert > 0$.
\end{proof}

\section{Global Convergence Results}\label{Sec:GlobalConvergence}

This section presents global convergence results which are in the 
same spirit as those known for trust-region-type methods.

The first result states that the inexactness criterion \eqref{eq_inexactness} is feasible, which implies that Algorithm~\ref{algo_proxRegNewton_cvx}
is well-defined.

\begin{lemma}
\label{lemma_inexactness}
For every $k \in \N_0$ such that $x^k$ is not a stationary point of \eqref{problem}, the inexactness criterion \eqref{eq_inexactness} is satisfied for any $x \in \dom \varphi$ sufficiently close to the exact solution $\overline x^k$ of \eqref{reg_proximal_newton_cvx}.
\end{lemma}

\begin{proof}
Recall that there are two criteria in \eqref{eq_inexactness}. We show
that both of them hold for all $ x $ sufficiently close to the global 
minimum of the underlying subproblem. Hence, consider a fixed iteration
index $k \in \N_0$ and assume that $x^k$ is not already a stationary point
of the given composite optimization problem \eqref{problem}. Since $\Vert R_k(\overline x^k) \Vert = 0$, it follows from the continuity of $R_k$ relative to $\dom \varphi$ that
\[
\Vert R_k(x) \Vert \leq \theta \min\left\{ \Vert r(x^k) \Vert, \Vert r(x^k) \Vert^{1+\tau} \right\}
\]
holds for $x \in \dom \varphi$ sufficiently close to $\overline x^k$,
showing that the first test in \eqref{eq_inexactness} holds for these
$ x $. Furthermore, from \cite[Proposition 2.4]{lee_yuekai_saunders_2014}, it follows that the exact solution $\overline x^k$ of subproblem \eqref{reg_proximal_newton_cvx} satisfies 
\begin{equation}
\label{eq_lemma_pred_ineq_1}
\nabla f(x^k)^\top  (\overline x^k-x^k)+\varphi(\overline x^k)-\varphi(x^k) \leq -(\overline x^k-x^k)^\top  G_k (\overline x^k-x^k).
\end{equation}
Therefore we obtain 
\begin{align}
\label{eq_lemma_pred_1}
\begin{split}
F(x^k)-\hat q_k(\overline x^k) &= \varphi(x^k)-\nabla f(x^k)^\top (\overline x^k-x^k)-\frac{1}{2}(\overline x^k-x^k)^\top G_k(\overline x^k-x^k)-\varphi(\overline x^k) \\
&= -\left(\nabla f(x^k)^\top  (\overline x^k-x^k)+\varphi(\overline x^k) - \varphi(x^k)\right)-\frac{1}{2}(\overline x^k-x^k)^\top G_k(\overline x^k-x^k) \\
&\geq \frac{1}{2} (\overline x^k - x^k)^\top G_k (\overline x^k-x^k) \geq \frac{\mu_k}{2}\Vert \overline x^k - x^k \Vert^2
> \frac{\alpha\mu_k}{2}\Vert \overline x^k - x^k \Vert^2,
\end{split}
\end{align}
where the first inequality follows from \eqref{eq_lemma_pred_ineq_1} and the second from the positive semidefiniteness of $H_k$. From the continuity of $F(x^k)-\hat q_k(\cdot)-\frac{\alpha\mu_k}{2}\Vert \cdot - x^k \Vert^2$ relative to $\dom \varphi$, it follows that
\[
F(x^k)-\hat q_k(x) > \frac{\alpha\mu_k}{2}\Vert x-x^k \Vert^2
\]
holds for all $x \in \dom \varphi$ sufficiently close to $\overline x^k$.

\end{proof}

\noindent 
The next result provides a lower and upper bound of the residual
$ r(x^k) $ in terms of the vector $ d^k $.

\begin{lemma}
\label{lemma_rk_dk}
For all $k \in \N_0$, it holds that
\[
\frac{\mu_k}{(1+\Vert G_k \Vert)(1+\theta)}\Vert d^k \Vert \leq \Vert r(x^k) \Vert \leq \frac{1+\Vert G_k \Vert}{1-\theta}\Vert d^k \Vert.
\]
\end{lemma}
\begin{proof}
By $r(x^k)=x^k-\prox_\varphi(x^k-\nabla f(x^k))$, we get from \eqref{eq_prox_subdiff} that $r(x^k) \in \nabla f(x^k) + \partial \varphi \left(x^k-r(x^k)\right)$.
In the same way, $R_k(\hat x^k) \in \nabla f(x^k) + G_kd^k+\partial \varphi(\hat x^k-R_k(\hat x^k))$ follows from the definition of the proximal operator. The monotonicity of the subgradient mapping $\partial\varphi$ ensures that
\begin{align}
\label{eq_lemma_rk_dk_1}
\left\langle R_k(\hat x^k)-r(x^k) - G_kd^k, \, d^k+r(x^k) - R_k(\hat x^k) \right\rangle \geq 0.
\end{align}
Simply reordering the left-hand side yields
\begin{align*}
0 \leq -\Vert r(x^k) \Vert^2 - \Vert R_k(\hat x^k) \Vert^2 + 2\langle R_k(\hat x^k),r(x^k) \rangle - (d^k)^\top G_k d^k + \langle R_k(\hat x^k)-r(x^k), d^k+G_kd^k \rangle .
\end{align*}
This implies
\begin{align*}
\Vert r(x^k) - R_k(\hat x^k) \Vert^2 &\leq \Vert r(x^k) \Vert^2-2\left\langle R_k(\hat x^k),r(x^k) \right\rangle + \Vert R_k(\hat x^k) \Vert^2 + (d^k)^\top G_kd^k \\
&\leq \left\langle R_k(\hat x^k)-r(x^k), \, d^k+G_kd^k \right\rangle \\
&\leq \Vert r(x^k) - R_k(\hat x^k) \Vert \cdot (1+\Vert G_k \Vert)\Vert d^k \Vert.
\end{align*}
Together with the inexactness criterion $ \Vert R_k(\hat x^k) \Vert \leq \theta \Vert r(x^k) \Vert$ and the Cauchy-Schwarz inequality, this results in 
\[
\Vert r(x^k) \Vert \leq \Vert r(x^k) - R_k(\hat x^k) \Vert + \Vert R_k(\hat x^k) \Vert \leq (1+\Vert G_k \Vert) \Vert d^k \Vert + \theta \Vert r(x^k) \Vert.
\]
Remembering $\theta \in (0,1)$, we get the upper estimate
\[
\Vert r(x^k) \Vert \leq \frac{1+\Vert G_k \Vert}{1-\theta}\Vert d^k \Vert.
\]
Reordering \eqref{eq_lemma_rk_dk_1} in a different way yields
\[
\langle d^k, G_k d^k \rangle \leq \langle R_k(\hat x^k) - r(x^k), d^k-R_k(\hat x^k) + r(x^k) + G_k d^k \rangle \leq \langle (I+G_k) d^k, R_k(\hat x^k) - r(x^k) \rangle .
\]
Using $G_k \succeq \mu_k I$ and the Cauchy-Schwarz inequality, we therefore get
\[
\mu_k \Vert d^k \Vert^2 \leq \langle d^k, G_k d^k \rangle \leq (1+ \Vert G_k \Vert) \Vert d^k \Vert \Vert R_k(\hat x^k) - r(x^k) \Vert \leq (1+\Vert G_k \Vert) \Vert d^k \Vert (1+\theta) \Vert r(x^k) \Vert,
\]
where the last inequality follows from \eqref{eq_inexactness}. Hence, dividing by $\mu_k \Vert d^k \Vert$ (in the case of $\Vert d^k \Vert = 0$, the resulting inequality holds trivially) yields
\[
\Vert d^k \Vert \leq \frac{(1+\theta)(1+\Vert G_k \Vert)}{\mu_k} \Vert r(x^k) \Vert.
\]
This completes the proof.
\end{proof}

\noindent
The following result provides (implicitly) a condition under which the quotient
between the actual and the predicted reduction is greater than a suitable
constant (note that, in the following, we often exploit the observation
from Remark~\ref{Rem:Positive}
that the predicted reduction is a positive number, without explicitly mentioning this fact).

\begin{lemma}
Let $c \leq 1$. For every $k \geq 0$, there exists $\xi^k$ on the line segment between $x^k$ and $\hat x^k$ such that
\begin{equation}
\text{ared}_k-c\,\text{pred}_k \geq \frac{1}{2}\left((1-c)\mu_k - \Vert \nabla^2f(\xi^k)-\nabla^2 f(x^k) \Vert\right)\Vert d^k \Vert^2.
\end{equation}
\label{lemma_ared_pred_ineq_cvx}
\end{lemma}

\begin{proof}
It follows from Taylor's formula and the convexity of $\dom\varphi$ that, for every $k\geq 0$, there exists $\xi^k \in \dom\varphi$ on the line segment between $x^k$ and $\hat x^k$ such that 
\[
f(\hat x^k)-f(x^k)-\nabla f(x^k)^\top  d^k=\frac{1}{2}(d^k)^\top \nabla^2f(\xi^k)d^k.
\] 
This yields
\begin{align*}
F(\hat x^k)-q_k(\hat x^k) &= f(\hat x^k)-f(x^k)-\nabla f(x^k)^\top  d^k-\frac{1}{2}(d^k)^\top  \nabla^2f(x^k)d^k \\
&= \frac{1}{2}(d^k)^\top (\nabla^2f(\xi^k)-\nabla^2 f(x^k))d^k \\
&\leq \frac{1}{2}\Vert \nabla^2f(\xi^k)-\nabla^2 f(x^k) \Vert\Vert d^k \Vert^2.
\end{align*}
Using this inequality together with \eqref{eq_pred}, we get
\begin{align*}
\text{ared}_k-c\,\text{pred}_k&= (1-c)\text{pred}_k-\text{pred}_k+\text{ared}_k = (1-c)\text{pred}_k-\left(F(\hat x^k)-q_k(\hat x^k)\right) \\
&\geq \frac{1-c}{2}\mu_k\Vert d^k \Vert^2-\frac{1}{2}\Vert \nabla^2f(\xi^k)-\nabla^2f(x^k) \Vert \Vert d^k \Vert^2 \\
&= \frac{1}{2}\left((1-c)\mu_k-\Vert \nabla^2f(\xi^k)-\nabla^2f(x^k) \Vert\right)\Vert d^k \Vert^2.
\end{align*}
This completes the proof.
\end{proof}

\noindent 
We next show that Algorithm~\ref{algo_proxRegNewton_cvx} generates
infinitely many successful or highly successful iterates.

\begin{theorem}
Suppose $\Vert r(x^k) \Vert \neq 0$ for all $k \geq 0$. Then Algorithm \ref{algo_proxRegNewton_cvx} performs infinitely many successful or highly successful iterations.
\label{theorem_well_defined_cvx}
\end{theorem}

\begin{proof}
Suppose there exists $k_0\geq 0$ such that all iterations $k \geq k_0$ are unsuccessful. Then at least one of the inequalities
\begin{equation}\label{Eq:In-violated}
\rho_k \leq c_1, \quad \text{pred}_k \leq p_{min}
(1 - \theta)\Vert d^k \Vert \min\{\Vert r(x^k) \Vert, \Vert r(x^k) \Vert^\kappa\}
\end{equation}
has to hold for all $k \geq k_0$. We will derive a contradiction and
show that both inequalities are eventually violated.

First note that Lemma \ref{lemma_alg_prop_mu}\ref{eq_lemma_alg_prop_mu_unsucc} implies $x^k=x^{k_0}$ for all $k \geq k_0$ and $\{\mu_k\} \to \infty$
whereas both $\{ \Vert r(x^k) \Vert \}$ and $\{\Vert H_k \Vert \}$ are bounded. Thus, remembering $\Vert G_k \Vert = \Vert H_k \Vert + \mu_k$, it follows from the first inequality in Lemma \ref{lemma_rk_dk} that $\{\Vert d^k \Vert\}$ is bounded by some $\overline d > 0$. For all $k \geq k_0$ it then holds that $\xi^k$ (from Lemma \ref{lemma_ared_pred_ineq_cvx}) belongs to the compact set $B_{\overline d}(x^{k_0}) \cap \Omega$ (recall that $ \Omega $ was supposed
to be a closed set). From the continuity of $\nabla ^2f(\cdot)$ on $\Omega$ it then follows that
\begin{equation}
\Vert \nabla^2f(\xi^k)-\nabla^2 f(x^k) \Vert < (1-c_1)\mu_k
\end{equation}
for sufficiently large $k\geq k_0$, which together with Lemma \ref{lemma_ared_pred_ineq_cvx} guarantees
\[
\text{ared}_k-c_1\text{pred}_k > 0,
\]
and therefore $\rho_k > c_1$, thus violating the first inequality
in \eqref{Eq:In-violated}. 

The second inequality in Lemma \ref{lemma_rk_dk} ensures that $\Vert d^k \Vert > 0$ for all $k \geq 0$. Thus, from Lemma \ref{lemma_rk_dk}, we get
\[
\frac{\Vert r(x^k) \Vert}{\Vert d^k \Vert\mu_k} \leq \frac{1+\Vert G_k \Vert}{(1-\theta)\mu_k} \leq \frac{1+\Vert H_k \Vert + \mu_k}{(1-\theta)\mu_k}
\]
for all $k \geq k_0$. Taking $k \to \infty$, it follows that the expression on the right-hand side tends to $1/(1-\theta)$. Hence, for $k \geq k_0$ sufficiently large it holds that
\[
\frac{\Vert r(x^k) \Vert}{\Vert d^k \Vert \mu_k} < \frac{1}{2p_{min}(1-\theta)}.
\]
This inequality together with \eqref{eq_pred} then yields
\begin{align}
\begin{split}
\label{eq_theorem_well_defined_cvx_1}
\text{pred}_k \geq \frac{\mu_k}{2}\Vert d^k \Vert^2 > p_{min}(1-\theta)\Vert r(x^k) \Vert \Vert d^k \Vert \geq p_{min}(1-\theta)\Vert d^k \Vert \min\{\Vert r(x^k) \Vert,\Vert r(x^k) \Vert^\kappa\}
\end{split}
\end{align}
for sufficiently large $k \geq k_0$, which contradicts the second
inequality in \eqref{Eq:In-violated}.
\end{proof}

\noindent 
We next present our first global convergence result for Algorithm~\ref{algo_proxRegNewton_cvx}.

\begin{theorem}
\label{global_convergence_cvx}
The sequence $\{x^k\}$ generated by Algorithm~\ref{algo_proxRegNewton_cvx} satisfies $\liminf_{k \to \infty} \Vert r(x^k) \Vert = 0$. 
\end{theorem}

\begin{proof}
Let $\mathcal S \subset \N$ be the set of successful or highly successful iterations, and recall that this set is infinite due to
Theorem~\ref{theorem_well_defined_cvx}. Assume, by contradiction, that $\liminf_{k \to \infty}\Vert r(x^k) \Vert > 0$. Then there exists $\varepsilon > 0$ such that $\min\{\Vert r(x^k) \Vert, \Vert r(x^k) \Vert^\kappa\} \geq \varepsilon$ for all $k \geq 0$. Lemma \ref{lemma_equiv} implies that the set $\mathcal K$ is finite, hence the set $\overline{\mathcal S}:=  \mathcal S \setminus \mathcal K$ is still infinite. By definition, it holds for all $k \in \overline{\mathcal S}$ that
\begin{align*}
F(x^k)-F(\hat x^k) &= \text{ared}_k > c_1 \text{pred}_k > c_1 p_{min} (1-\theta) \Vert d^k \Vert \min\{\Vert r(x^k) \Vert, \Vert r(x^k) \Vert^\kappa\} \\
&\geq c_1 p_{min} (1-\theta) \Vert d^k \Vert \varepsilon,
\end{align*}
cf. Lemma~\ref{lemma_alg_prop_mu}.
Since $F$ is bounded from below, summation yields
\[
\infty > \sum_{k=0}^{\infty} [F(x^k)-F(x^{k+1})] \geq \sum_{k \in 
\overline{\mathcal S}} [F(x^k)-F(\hat x^k)] \geq c_1p_{min}(1-\theta)\varepsilon\sum_{k \in \overline{\mathcal S}}\Vert d^k \Vert
\]
(where we used the fact that $ F(x^k)-F(x^{k+1}) \geq 0 $ for all $ k $).
Taking into account that $x^k$ is not updated in unsuccessful steps, it follows that
\begin{equation}\label{Eq:series-conv}
\infty > \sum_{k \in \overline{\mathcal S}} \Vert d^k \Vert + \sum_{k \in \mathcal K} \Vert d^k \Vert = \sum_{k \in \mathcal S} \Vert d^k \Vert = \sum_{k \in \mathcal S} \Vert x^{k+1}-x^k \Vert = \sum_{k=0}^\infty \Vert x^{k+1}-x^k \Vert,
\end{equation}
where we used the previous inequality and the finiteness of $\mathcal K$ in the first inequality. Hence, $\{x^k\}$ is a Cauchy sequence and therefore convergent to some $\overline x \in \R^n$. The mapping $x \mapsto \nabla^2f(x) + a[-\lambda_{min}(\nabla^2\psi(Ax-b))]_+A^\top A$ is continuous, i.e., the sequence $\{H_k\}$ is also convergent. Define $M := \sup\{ \Vert H_k \Vert \, \vert \, k \geq 0 \} < \infty$. Since $\Vert r(\cdot) \Vert$ is continuous, we have $\Vert r(\overline x) \Vert = \lim_{k \to \infty} \Vert r(x^k) \Vert \geq \varepsilon$ and $\overline x$ is not a stationary point of \eqref{problem}.  Using the boundedness of $\{H_k\}$ together with Lemma \ref{lemma_rk_dk} yields
\[
\Vert r(x^k) \Vert \leq \frac{1+M+\mu_k}{1-\theta}\Vert d^k \Vert.
\]
Note that \eqref{Eq:series-conv} implies
$\Vert d^k \Vert \to_{\mathcal S} 0$.
If there were a subset $\mathcal S' \subseteq \mathcal S$ such that $\{\mu_k\}_{\mathcal S'}$ is bounded, then $\{\Vert r(x^k) \Vert\}_{\mathcal S'}$ would converge to zero, a contradiction. Hence, $\{\mu_k\} \to_{\mathcal S} \infty$. Since $\mu_k$ can not decrease during unsuccessful iterations, it follows that $\{\mu_k\} \to \infty$. This implies that Algorithm \ref{algo_proxRegNewton_cvx} also performs infinitely many unsuccessful iterations. 

For every $k \geq 0$, Taylor's formula yields the existence of
a vector $\xi^k$ on the straight line between $x^k$ and $\hat x^k$ such that $f(\hat x^k)-f(x^k)=\nabla f(\xi^k)^\top d^k$. Note that, similar to the proof of Theorem \ref{theorem_well_defined_cvx}, $\{\Vert d^k \Vert\}$ is bounded. Hence, for some $\overline d > 0$ and $k$ sufficiently large, $\xi^k$ belongs to the compact set $B_{\overline d}(\overline x) \cap \Omega$. Note that $\nabla f$ is continuously differentiable and therefore also locally Lipschitz continuous, hence Lipschitz continuous on compact sets. In particular, there exists a constant $\overline L > 0$ such that
\begin{equation}
\label{eq_global_convergence_cvx_1}
\Vert \nabla f(\xi^k) - \nabla f(x^k) \Vert \leq \overline L \Vert \xi^k - x^k \Vert \leq \overline L \Vert d^k \Vert
\end{equation}
holds for $k$ sufficiently large. By using \eqref{eq_pred} in the first, Taylor's formula in the second and \eqref{eq_global_convergence_cvx_1} in the last inequality, we obtain
\begin{align*}
\vert \rho_k-1 \vert &= \left\vert\frac{\text{ared}_k}{\text{pred}_k}-1\right\vert 
=\left\vert\frac{\text{ared}_k-\text{pred}_k}{\text{pred}_k}\right\vert 
= \left\vert \frac{F(\hat x^k)-q_k(\hat x^k)}{\text{pred}_k} \right\vert \\
&\leq \frac{\vert f(\hat x^k)-f(x^k)-\nabla f(x^k)^\top d^k - \frac{1}{2}(d^k)^\top \nabla^2 f(x^k) d^k \vert}{\frac{1}{2} \mu_k \Vert d^k \Vert^2} \\
&\leq \frac{2\left\vert \nabla f(\xi^k)^\top d^k-\nabla f(x^k)^\top d^k \right\vert+\left\vert (d^k)^\top \nabla^2 f(x^k) d^k \right\vert}{\mu_k \Vert d^k \Vert^2} \\
&\leq \frac{2\Vert \nabla f(\xi^k)-\nabla f(x^k) \Vert \Vert d^k \Vert+ \Vert \nabla^2 f(x^k) \Vert \Vert d^k \Vert^2}{\mu_k \Vert d^k \Vert^2} \\
&\leq \frac{2\overline L+ \Vert \nabla^2 f(x^k) \Vert}{\mu_k} \to 0
\end{align*}
for $k \to \infty$. Hence, $\{\rho_k\} \to 1$, i.e., eventually all steps are highly successful, which yields a contradiction and therefore 
$ \label{global_conv_eq}
\liminf_{k \to \infty}\Vert r(x^k) \Vert = 0 $.
\end{proof}

\noindent
The following global convergence theorem is the same as \cite[Theorem 5.7]{lechner_2022}. Its proof is only slightly adapted to our case.

\begin{theorem}
Assume that $\nabla f$ is uniformly continuous on a set $\mathcal X$ satisfying $\{x^k\} \subset \mathcal X$. Then $\lim_{k \to \infty}\Vert r(x^k) \Vert=0$ holds. In particular, every accumulation point of $\{x^k\}$ is a stationary point of F.
\end{theorem}

\begin{proof}
Assume, by contradiction, that there exists $\varepsilon > 0$ and $\mathcal L \subset \N$ such that $\Vert r(x^k) \Vert \geq 2\varepsilon$ for all $k \in \mathcal L$. Set $ \overline{\varepsilon} := \min \{ \varepsilon,
\varepsilon^{\kappa} \} $.
By Theorem \ref{global_convergence_cvx}, for each $k \in \mathcal L$, there is an index $l_k > k$ such that $\Vert r(x^l) \Vert \geq \varepsilon$ for all $k \leq l < l_k$ and $\Vert r(x^{l_k}) \Vert < \varepsilon$. If, for $k \in \mathcal L$, an iteration $k \leq l < l_k$ is successful or highly successful, we get
\[
F(x^l)-F(x^{l+1}) \geq c_1 \text{pred}_l > 
c_1 (1 - \theta) p_{min} \Vert d^l \Vert \Vert r(x^l) \Vert 
\geq c_1 (1 - \theta) p_{min} \overline{\varepsilon} \Vert x^{l+1}-x^l \Vert.
\]
For unsuccessful iterations $l$, this estimate holds trivially. Thus,
\begin{align*}
(1 - \theta) p_{min}c_1
\overline{\varepsilon} \Vert x^{l_k}-x^k\Vert & \leq (1 - \theta)  p_{min}c_1\varepsilon\sum_{l=k}^{l_k-1}\Vert x^{l+1}-x^l\Vert \\
& \leq \sum_{l=k}^{l_k-1} F(x^l)-F(x^{l+1})=F(x^k)-F(x^{l_k})
\end{align*}
holds for all $k \in \mathcal L$. By Assumption~\ref{ass_inf_F}, $F$ is bounded from below, and by construction, the sequence $\{F(x^k)\}$ is monotonically decreasing, hence convergent. This implies that the sequence $\{F(x^k)-F(x^{l_k})\}_{\mathcal L}$ converges to 0. Hence, we get $\{\Vert x^{l_k}-x^k \Vert\}_{\mathcal L} \to 0$. The uniform continuity of $\nabla f$ and of the proximity operator together with the fact that the composition of uniformly continuous functions is uniformly continuous, yields the uniform continuity of the residual funciton $r(\cdot)$. Thus, we get $\{\Vert r(x^{l_k})-r(x^k) \Vert \}_{\mathcal L} \to 0$. On the other hand, by the choice of $l_k$, we have
\[
\Vert r(x^k)-r(x^{l_k}) \Vert \geq \Vert r(x^k) \Vert - \Vert r(x^{l_k}) \Vert \geq 2\varepsilon-\varepsilon=\varepsilon
\]
for all $k \in \mathcal L$, which yields the desired contradiction.
\end{proof}

\section{Local Superlinear Convergence}\label{Sec:LocalConv}

The aim of this section is to prove local fast superlinear convergence
of Algorithm~\ref{algo_proxRegNewton_cvx} under the following (fairly mild)
assumptions.

\begin{assumption}
\begin{enumerate}[label=(\alph*), ref=\ref{assumption_loc_prox_cvx}(\alph*)]
\item The set $X^*$ of strongly stationary points of \eqref{problem} is nonempty and there exists an accumulation point $x^* \in X^*$ of $\{x^k\}_{\mathcal K}$. \label{ass_existence_cvx}
\item $\nabla^2\psi$ is locally Lipschitz continuous at $Ax^*-b$ relative to $A(\dom\varphi)-b$, i.e., there exists $\varepsilon > 0$ and $L_\psi>0$ such that
\[
\Vert \nabla^2\psi(Ax-b)-\nabla^2\psi(Ay-b) \Vert \leq L_\psi\Vert Ax-Ay \Vert, \quad \forall x,y \in B_{\varepsilon}(x^*)  \cap \dom\varphi.
\]
\label{ass_lipschitz_cvx}
\item $\Vert r(x) \Vert$ provides a local Hölderian error bound for problem \eqref{problem} on $B_{\varepsilon}(x^*) \cap \dom \varphi$, i.e., there exist constants $\beta>0$ and $q > \max\{\delta,1-\delta\}$ such that
\begin{equation}
\beta\dist(x,X^*) \leq \Vert r(x) \Vert^q, \quad \forall x \in B_{\varepsilon}(x^*) \cap \dom\varphi,
\end{equation}
where $ \delta > 0 $ denotes the constant from Algorithm~\ref{algo_proxRegNewton_cvx}.
\label{ass_luo_tseng_errbound_cvx}
\end{enumerate}
\label{assumption_loc_prox_cvx}
\end{assumption}

\noindent
Note that Lemma~\ref{lemma_equiv} and Theorem~\ref{global_convergence_cvx} ensure that $\mathcal K$ is an infinite set. Hence, the subsequence $\{x^k\}_{\mathcal K}$ in Assumption~\ref{ass_existence_cvx} is well-defined. Define 
$$
   \varepsilon_0 := \min \big\{ \varepsilon, \varepsilon / \Vert A \Vert \big\}
   \leq \varepsilon,
$$ 
where $ \varepsilon > 0 $
denotes the radius from Assumption~\ref{ass_lipschitz_cvx}. For $x,y \in B_{\varepsilon_0}(x^*) \cap \dom\varphi$, it then follows from \eqref{eq_grad_hess_f} and Assumption \ref{ass_lipschitz_cvx}¸ that $\nabla^2f$ is locally Lipschitz continuous at $x^*$ relative to $\dom\varphi$ with Lipschitz constant $L := \Vert A \Vert^3 L_\psi$, i.e. 
\begin{equation}
\label{eq_locLip_hess_f}
\Vert \nabla^2 f(x) - \nabla^2 f(y) \Vert \leq L \Vert x - y \Vert, \quad \forall x,y \in B_{\varepsilon_0}(x^*).
\end{equation}
This, in turn, implies that
\begin{equation}
\Vert \nabla f(x)-\nabla f(y)-\nabla^2f(x)(x-y) \Vert \leq \frac{L}{2}\Vert x-y \Vert^2, \quad \forall x,y \in B_{\varepsilon_0}(x^*).
\label{eq_LipContCorr_cvx}
\end{equation}
Furthermore, since
$f$ is twice continuously differentiable, $\nabla f$  is continuously differentiable and, therefore, locally Lipschitz continuous. Consequently,
there exists a constant $L_g>0$ such that
\begin{align}
\label{eq_LocLip}
\Vert \nabla f(x)-\nabla f(y) \Vert \leq L_g\Vert x - y \Vert, \quad \forall x,y \in B_{\varepsilon_0}(x^*).
\end{align}
In particular, we therefore have
\begin{align}
\label{eq_Hess_loc_bd}
\Vert \nabla^2f(x) \Vert \leq L_g, \quad \forall x \in B_{\varepsilon_0}(x^*).
\end{align}
In the following, for each $k \geq 0$, we denote by $\widetilde x^k$ a point satisfying the properties
\begin{equation}
\Vert x^k-\widetilde x^k \Vert = \dist(x^k,X^*), \quad \widetilde x^k \in X^*,
\end{equation}
i.e., $ \widetilde x^k $ is a (not necessarily unique) projection of $ x^k $
onto the nonempty and closed (not necessarily convex) set $ X^* $.

\begin{lemma}
Suppose that Assumptions \ref{assumption_loc_prox_cvx} hold. Then, for every iteration $k \geq 0$ with $x^k \in B_{\varepsilon_0 / 2}(x^*)$, it holds that
\[
\Vert r(x^k) \Vert \leq (2+L_g)\dist(x^k,X^*).
\]
\label{lemma_res_ineq}
\end{lemma}

\begin{proof}
First observe that 
\begin{align}
\label{eq_lemma_res_ineq_1}
\Vert \widetilde x^k-x^* \Vert \leq \Vert x^k-x^* \Vert + \Vert \widetilde x^k-x^k \Vert \leq 2 \Vert x^k-x^* \Vert,
\end{align}
i.e., for $x^k \in B_{\varepsilon_0/2}(x^*)$, it holds that $\widetilde x^k \in B_{\varepsilon_0}(x^*)$. Remembering the definition of $\widetilde x^k$, 
we obtain
\begin{align*}
\Vert r(x^k) \Vert&=\Vert r(x^k)-r(\widetilde x^k) \Vert \\
&= \Vert \prox_\varphi(x^k-\nabla f(x^k))-x^k-\prox_\varphi(\widetilde x^k-\nabla f(\widetilde x^k))+\widetilde x^k \Vert \\
&\leq \Vert \prox_\varphi(x^k-\nabla f(x^k))-\prox_\varphi(\widetilde x^k-\nabla f(\widetilde x^k)) \Vert+\Vert x^k-\widetilde x^k \Vert \\
&\leq \Vert x^k-\widetilde x^k-\nabla f(x^k)+\nabla f(\widetilde x^k) \Vert+\Vert x^k-\widetilde x^k \Vert \\
&\leq \Vert \nabla f(x^k)-\nabla f(\widetilde x^k) \Vert+2\Vert x^k-\widetilde x^k \Vert \\
&\leq (2+L_g)\dist(x^k,X^*),
\end{align*}
where the second inequality follows from the non-expansiveness of the proximity operator and the last inequality follows from \eqref{eq_LocLip}, taking
into account that $ x^k, \widetilde{x}^k \in B_{\varepsilon_0} (x^*) $.
\end{proof}

\noindent
The following lemma is almost identical to \cite[Lemma 4.2]{liu_pan_wu_yang_2024}. For the convenience of our readers, its proof is provided here, with slight adaptations to our case.

\begin{lemma}
\label{lemma_loc_inexact}
For each $k \in \mathcal K$, it holds that $\Vert \hat x^k-\overline x^k \Vert \leq \nu_{min}^{-1} \theta (1+\Vert G_k \Vert)\Vert r(x^k) \Vert^{1+\tau-\delta}$.
\end{lemma}

\begin{proof}
Consider a fixed index $k \in \mathcal K$. From the definition of $R_k(\hat x^k)$ and relation \eqref{eq_prox_subdiff}, it follows that 
\begin{align*}
& \hspace*{-15mm}\hat x^k - R_k(\hat x^k) \in \hat x^k - \nabla f(x^k) - G_k d^k - \partial \varphi (\hat x^k-R_k(\hat x^k)) \\
\Longleftrightarrow \; & R_k(\hat x^k) - \nabla f(x^k) - G_k d^k \in \partial\varphi (\hat x^k - R_k(\hat x^k)).
\end{align*}
Since $\overline x^k$ is the exact solution of \eqref{reg_proximal_newton_cvx} it holds by Fermat's theorem that 
\[
-\nabla f(x^k) - G_k(\overline x^k-x^k) \in \partial \varphi (\overline x^k).
\]
By the monotonicity of $\partial\varphi$, we have
\begin{align*}
\langle R_k(\hat x^k) - G_k(\hat x^k-\overline x^k), \hat x^k - R_k(\hat x^k) - \overline x^k \rangle \geq 0.
\end{align*}
Reordering yields
\begin{align*}
\langle \hat x^k - \overline x^k, G_k(\hat x^k-\overline x^k) \rangle &\leq \langle R_k(\hat x^k), \hat x^k - \overline x^k - R_k(\hat x^k) + G_k(\hat x^k-\overline x^k) \rangle \\
&\leq \langle R_k(\hat x^k), (I+G_k) (\hat x^k-\overline x^k) \rangle.
\end{align*}
Combining this inequality with $G_k \succeq \mu_kI$ and using \eqref{eq_inexactness} yields 
\[
\mu_k \Vert \hat x^k - \overline x^k \Vert^2 \leq (1+\Vert G_k \Vert ) \Vert R_k(\hat x^k) \Vert \Vert \hat x^k-\overline x^k \Vert \leq \theta (1+\Vert G_k \Vert)\Vert r(x^k) \Vert^{1+\tau} \Vert \hat x^k - \overline x^k \Vert.
\]
Dividing by $\mu_k\Vert \hat x^k - \overline x^k \Vert$ (the case $\Vert \hat x^k-\overline x^k \Vert = 0$ is trivial) and using Lemma \ref{lemma_alg_prop_K}\ref{eq_lemma_alg_prop_K_eq} along with $\nu_k \geq \nu_{min}$ demonstrates that the desired result holds.
\end{proof}

\noindent
The following lemma is identical to \cite[Lemma 4.4]{liu_pan_wu_yang_2024}. Again, its proof is presented here, only adapting the notation to our case.

\begin{lemma}
\label{lemma_Lambda}
Suppose that Assumptions \ref{assumption_loc_prox_cvx} hold. Then for every $k \geq 0$ with $x^k \in B_{\varepsilon_0/2}(x^*)$ it holds that
\[
\Lambda_k \leq aL_\psi \Vert A \Vert \dist(x^k,X^*).
\]
\end{lemma}

\begin{proof}
Let $x^k \in B_{\varepsilon_0/2}(x^*)$ be fixed. By definition of $\Lambda_k$, it suffices to consider the case where $\lambda_{min}(\nabla^2 \psi(Ax^k-b)) < 0$. In view of \eqref{eq_lemma_res_ineq_1}, we obtain $\Vert \widetilde x^k - x^* \Vert \leq \varepsilon_0$, and consequently $\widetilde x^k \in B_{\varepsilon}(x^*) \cap \dom \varphi$. From $\widetilde x^k \in X^*$, we have $\nabla^2\psi(A\widetilde x^k - b) \succeq 0$. When $\lambda_{min}(\nabla^2\psi(A \widetilde x^k - b)) = 0$, then 
\begin{align*}
\Lambda_k &= -a \lambda_{min} (\nabla^2 \psi(Ax^k-b)) = a [\lambda_{min}(\nabla^2 \psi(A \widetilde x^k - b)) - \lambda_{min}(\nabla^2\psi(Ax^k-b))] \\
&\leq a \Vert \nabla^2 \psi(A\widetilde x^k-b) - \nabla^2 \psi(Ax^k-b) \Vert \leq a L_\psi \Vert A \Vert \Vert x^k - \widetilde x^k \Vert, 
\end{align*}
where the first inequality is by the Lipschitz continuity of the function $\mathbb S^n \ni Z \mapsto \lambda_{min}(Z)$ with modulus $1$ (follows from Weyl's inequality), and the second one is using 
Assumption~\ref{ass_lipschitz_cvx}. So we only need to consider the case $\lambda_{min}(\nabla^2 \psi(A \widetilde x^k - b)) > 0$. For this purpose, let $\phi_k(t) := \lambda_{min}[\nabla^2 \psi(Ax^k-b+tA(\widetilde x^k-x^k))]$ for $t \geq 0$. Clearly, $\phi_k$ is continuous on any open interval containing $[0,1]$. Note that $\phi_k(0) < 0$ and $\phi_k(1) > 0$. 
Hence, there exists $\overline t_k \in (0,1)$ such that $\phi_k(\overline t_k) = 0$. Consequently,
\begin{align*}
\Lambda_k &= -a \lambda_{min} (\nabla^2 \psi(Ax^k-b)) \\
&= a[\lambda_{min}(\nabla^2 \psi(Ax^k-b+\overline t_k A(\widetilde x^k - x^k))) - \lambda_{min}(\nabla^2 \psi(Ax^k-b))] \\
&\leq a \Vert \nabla^2 \psi(Ax^k-b+\overline t_k A(\widetilde x^k-x^k))-\nabla^2\psi(Ax^k-b) \Vert \leq a L_\psi \Vert A \Vert \Vert \widetilde x^k- x^k \Vert. 
\end{align*}
This shows that the desired result holds.
\end{proof}

\begin{lemma}
Suppose that Assumption~\ref{assumption_loc_prox_cvx} holds. Define $\varepsilon_1 := \min\big\{ \frac{1}{2+L_g},\frac{\varepsilon_0}{2}\big\}$. Then, for $k \in \mathcal K$ with $x^k \in B_{\varepsilon_1}(x^*)$, it holds that 
\[
\Vert d^k \Vert \leq c\dist(x^k,X^*),
\]
where $c:=\nu_{min}^{-1} \theta (2+L_g)^{1+\tau-\delta} (1+L_g+aL + \overline\nu (2+L_g)^\delta) + \frac{L+2aL}{2\nu_{min}\beta} + 2$.
\label{lemma_d^k_ineq_cvx}
\end{lemma}

\begin{proof}
Let $k \in \mathcal K$ and $x^k \in B_{\varepsilon_1}(x^*)$ be fixed. From the definition of $\widetilde x^k$ it follows that $0 \in \nabla f(\widetilde x^k)+\partial\varphi(\widetilde x^k)$ and thus
\begin{equation}
\nabla f(x^k)-\nabla f(\widetilde x^k)+(H_k+\mu_kI)(\widetilde x^k-x^k) \in \nabla f(x^k)+(H_k+\mu_kI)(\widetilde x^k-x^k)+\partial\varphi(\widetilde x^k).
\end{equation}
Together with
\begin{equation}
0 \in \nabla f(x^k)+(H_k+\mu_kI)(\overline x^k-x^k)+\partial\varphi(\overline x^k)
\end{equation}
it follows from the strong monotonicity of the mapping $\nabla f(x^k)+(H_k+\mu_kI)(\cdot-x^k)+\partial\varphi(\cdot)$ on $\R^n$ that
\begin{equation}
\label{eq_lemma_d^k_ineq_cvx_1}
\left\langle \nabla f(x^k)-\nabla f(\widetilde x^k)+(H_k+\mu_kI)(\widetilde x^k-x^k), \, \widetilde x^k-\overline x^k\right\rangle \geq \mu_k\Vert \widetilde x^k-\overline x^k\Vert^2.
\end{equation}
As in \eqref{eq_lemma_res_ineq_1} it holds that $\widetilde x^k \in B_{\varepsilon_0}(x^*)$ and from Lemma \ref{lemma_res_ineq} it follows that $\Vert r(x^k) \Vert \leq 1$. We now get
\begin{align}
\label{eq_lemma_dk_1}
\begin{split}
\Vert \overline x^k-x^k \Vert &= \Vert \overline x^k-\widetilde x^k+\widetilde x^k - x^k \Vert 
\leq \Vert \overline x^k-\widetilde x^k \Vert + \Vert \widetilde x^k-x^k \Vert \\
&\leq \frac{1}{\mu_k}\Vert \nabla f(x^k)-\nabla f(\widetilde x^k)+(H_k+\mu_kI)(\widetilde x^k-x^k)\Vert+\Vert \widetilde x^k-x^k \Vert \\
&\leq \frac{1}{\mu_k} \left( \Vert \nabla f(x^k)-\nabla f(\widetilde x^k)+H_k(\widetilde x^k-x^k)\Vert\right)+2\Vert \widetilde x^k-x^k \Vert \\
&\leq \frac{1}{\mu_k}\left( \frac{L}{2}\Vert \widetilde x^k-x^k \Vert^2 + \Lambda_k \Vert A^2 \Vert \Vert \widetilde x^k-x^k \Vert\right) + 2\Vert \widetilde x^k-x^k \Vert \\
&\leq \frac{L+2aL_\psi\Vert A \Vert^3}{2\mu_k}\dist(x^k,X^*)^2 + 2\dist(x^k,X^*)  \\
&= \frac{L+2aL}{2\nu_k\Vert r(x^k) \Vert^\delta}\dist(x^k,X^*)^2 + 2\dist(x^k,X^*) \\
&\leq \frac{L+2aL}{2\nu_k\Vert r(x^k) \Vert^q}\dist(x^k,X^*)^2 + 2\dist(x^k,X^*) \\
&\leq \frac{L+2aL}{2\nu_{min}\beta\dist(x^k,X^*)}\dist(x^k,X^*)^2 + 2\dist(x^k,X^*) \\
&= \left(\frac{L+2aL}{2\nu_{min}\beta} + 2\right)\dist(x^k,X^*), 
\end{split}
\end{align}
where we used \eqref{eq_lemma_d^k_ineq_cvx_1} together with the Cauchy-Schwarz inequality in the second, the triangle inequality and \eqref{eq_LipContCorr_cvx} in the fourth, Lemma \ref{lemma_Lambda} and the definition of $\widetilde x^k$ in the fifth, $q \geq \delta$ together with $\Vert r(x^k) \Vert \leq 1$ in the sixth, and 
Assumption~\ref{ass_luo_tseng_errbound_cvx} in the seventh inequality. In the second equality we used Lemma \ref{lemma_alg_prop_K}\ref{eq_lemma_alg_prop_K_eq}. Since $k \in \mathcal K$, it holds that
\begin{align*}
\Vert G_k \Vert &\leq \Vert \nabla^2f(x^k) \Vert + \Lambda_k \Vert A^\top A \Vert + \mu_k \leq L_g + aL\dist(x^k,X^*) + \overline\nu \Vert r(x^k) \Vert^\delta \\
&\leq L_g+aL\dist(x^k,X^*) + \overline\nu (2+L_g)^\delta \dist(x^k,X^*)^\delta \\
&\leq L_g+aL + \overline\nu (2+L_g)^\delta,
\end{align*}
where we used the triangle inequality in the first, \eqref{eq_Hess_loc_bd}, Lemma \ref{lemma_Lambda} and Lemma \ref{lemma_alg_prop_K}\ref{eq_lemma_alg_prop_K_mu} in the second, Lemma \ref{lemma_res_ineq} in the third, and $\dist(x^k,X^*) \leq 1$
(simply because $ x^k \in B_{\varepsilon_1} (x^*) $ and 
$ \varepsilon_1 < 1 $) in the last inequality.
We now obtain
\begin{align*}
\Vert d^k \Vert &= \Vert \hat x^k-\overline x^k + \overline x^k - x^k \Vert \leq \Vert \hat x^k - \overline x^k \Vert + \Vert \overline x^k-x^k \Vert \\
&\leq \nu_{min}^{-1} \theta (1+\Vert G_k \Vert)\Vert r(x^k) \Vert^{1+\tau-\delta} + \Vert \overline x^k-x^k \Vert \\
&\leq \nu_{min}^{-1} \theta (1+L_g+aL + \overline\nu (2+L_g)^\delta)(2+L_g)^{1+\tau-\delta}\dist(x^k,X^*) + \Vert \overline x^k-x^k \Vert \\
&\leq c \dist(x^k,X^*),
\end{align*}
where we used Lemma \ref{lemma_loc_inexact} in the second, Lemma \ref{lemma_res_ineq}, $\dist(x^k,X^*) \leq 1$, $\tau \geq \delta$ and the previous inequality in the third, and \eqref{eq_lemma_dk_1} in the last inequality.
\end{proof}

\begin{lemma}
\label{lemma_superlinear}
Suppose that Assumption~\ref{assumption_loc_prox_cvx} holds. Define $\varepsilon_2 := \min\big\{\frac{1}{2+L_g},\frac{1}{aL_\psi\Vert A \Vert},\frac{\varepsilon_0}{1+c}\big\}$, where $ c > 0 $ is the 
constant from Lemma~\ref{lemma_d^k_ineq_cvx}. For $k \in \mathcal K$ with $x^k \in B_{\varepsilon_2}(x^*)$, it then holds that
\begin{equation}
\label{eq_lemma_superlinear_1}
\Vert r(\hat x^k) \Vert \leq \hat c \Vert r(x^k) \Vert^{\min\{\delta+q,1+\tau\}},
\end{equation}
\begin{equation}
\label{eq_lemma_superlinear_2}
\dist(\hat x^k,X^*) \leq \widetilde c\dist(x^k,X^*)^{(1+\delta)q},
\end{equation}
with constants $\hat c$ and $\widetilde c$ defined by
\begin{align*}
\hat c &:= \frac{c^2L+2 a c L_{\psi} \Vert A \Vert^3+2\beta c\overline\nu }{2\beta^2}+\theta, \\
\widetilde c &:= \frac{1}{\beta}\left(\frac{c^2L}{2}+ a c L_{\psi}\Vert A \Vert^3+c\overline\nu (2+L_g)^\delta + \theta (2+L_g)^{1+\tau}\right)^q.
\end{align*}
\end{lemma}

\begin{proof}
Using the definition of $\varepsilon_2$ as well as Lemmas~\ref{lemma_res_ineq} and \ref{lemma_Lambda}, it follows that $\dist(x^k,X^*) \leq 1$, $\Vert r(x^k) \Vert \leq 1$ and $\Lambda_k \leq 1$ whenever $x^k \in B_{\varepsilon_2}(x^*)$. Additionally, for $x^k \in B_{\varepsilon_2}(x^*) \subseteq B_{\varepsilon_1}(x^*)$, it follows from Lemma \ref{lemma_d^k_ineq_cvx} that 
\[
\Vert \hat x^k-x^* \Vert \leq \Vert x^k-x^* \Vert + \Vert d^k \Vert \leq (1+c)\Vert x^k-x^* \Vert \leq \varepsilon_0,
\]
i.e., $\hat x^k \in B_{\varepsilon_0}(x^*)$. We now get
\begin{align*}
\Vert r(\hat x^k) \Vert &= \Vert \prox_\varphi(\hat x^k-\nabla f(\hat x^k))-\hat x^k \Vert \\
&= \Vert \prox_\varphi(\hat x^k-\nabla f(\hat x^k))- \prox_\varphi(\hat x^k-\nabla f(x^k)-(H_k+\mu_kI)d^k) - R_k(\hat x^k)\Vert \\
&\leq \Vert \prox_\varphi(\hat x^k-\nabla f(\hat x^k))- \prox_\varphi(\hat x^k-\nabla f(x^k)-(H_k+\mu_kI)d^k) \Vert + \Vert R_k(\hat x^k) \Vert \\
&\leq \Vert \nabla f(\hat x^k)-\nabla f(x^k)-(H_k+\mu_kI)d^k\Vert + \Vert R_k(\hat x^k) \Vert \\
&\leq \Vert \nabla f(\hat x^k)-\nabla f(x^k)-\nabla^2f(x^k)d^k \Vert + \Lambda_k \Vert A^\top A d^k \Vert + \mu_k \Vert d^k \Vert + \Vert R_k(\hat x^k) \Vert \\
&\leq \frac{L}{2}\Vert d^k \Vert^2 + \Lambda_k \Vert A \Vert^2 \Vert d^k \Vert + \mu_k \Vert d^k \Vert + \Vert R_k(\hat x^k) \Vert \\
&\leq \frac{c^2L}{2}\dist(x^k,X^*)^2 + a c L_{\psi} \Vert A \Vert^3 \dist(x^k,X^*)^2 + c\mu_k\dist(x^k,X^*) + \Vert R_k(\hat x^k) \Vert \\
&\leq \frac{c^2L}{2\beta^2}\Vert r(x^k) \Vert^{2q} + \frac{a c L_{\psi}\Vert A \Vert^3}{\beta^2}\Vert r(x^k) \Vert^{2q} + \frac{c\overline\nu }{\beta}\Vert r(x^k) \Vert^{\delta+q} + \theta\Vert r(x^k) \Vert^{1+\tau} \\
&\leq \left(\frac{c^2L+2 a c L_{\psi} \Vert A \Vert^3 +2\beta c\overline\nu }{2\beta^2}+\theta\right)\Vert r(x^k) \Vert^{\min\{\delta+q,1+\tau\}},
\end{align*}
where we used the nonexpansiveness in the second, \eqref{eq_LipContCorr_cvx} and $\Lambda_k \leq 1$ in the fourth, Lemma~\ref{lemma_Lambda} and 
Lemma~\ref{lemma_d^k_ineq_cvx} in the fifth, 
Assumption~\ref{ass_luo_tseng_errbound_cvx}, 
Lemma~\ref{lemma_alg_prop_K}\ref{eq_lemma_alg_prop_K_mu} and the inexactness criterion \eqref{eq_inexactness} in the sixth, and $q \geq \delta$ together with $\Vert r(x^k) \Vert \leq 1$ in the last inequality. Reusing the fifth inequality from above we also get
\begin{align*}
\Vert r(\hat x^k) \Vert &\leq \frac{c^2L}{2}\dist(x^k,X^*)^2 + 
a c L_{\psi} \Vert A \Vert^3 \dist(x^k,X^*)^2 + c\mu_k\dist(x^k,X^*) + \Vert R_k(\hat x^k) \Vert \\
&\leq \left(\frac{c^2L}{2}+ a c L_{\psi}\Vert A \Vert^3 +c\overline\nu (2+L_g)^\delta\right)\dist(x^k,X^*)^{1+\delta} + \theta \Vert r(x^k) \Vert^{1+\tau} \\
&\leq \left(\frac{c^2L}{2}+ a c L_{\psi}\Vert A \Vert^3 +c\overline\nu (2+L_g)^\delta + \theta (2+L_g)^{1+\tau}\right)\dist(x^k,X^*)^{1+\delta},   
\end{align*}
where we used Lemma \ref{lemma_alg_prop_K}\ref{eq_lemma_alg_prop_K_mu}, Lemma \ref{lemma_res_ineq}, $\dist(x^k,X^*) \leq 1$ and the inexactness criterion \eqref{eq_inexactness} in the second, and Lemma \ref{lemma_res_ineq} as well as $\tau \geq \delta$ in the third inequality. From 
Assumption~\ref{ass_luo_tseng_errbound_cvx} and the previous inequality, 
we then obtain
\begin{align*}
\dist(\hat x^k,X^*) &\leq \frac{1}{\beta}\Vert r(\hat x^k) \Vert^q \leq \widetilde c \dist(x^k,X^*)^{(1+\delta)q},
\end{align*}
and this completes the proof.
\end{proof}

\noindent
We finally present the main local rate-of-convergence result.

\begin{theorem}
Suppose that Assumption~\ref{assumption_loc_prox_cvx} holds. Then $\{x^k\}$ converges to $x^*$ and $\{\Vert r(x^k) \Vert\}$ converges to 0 at the rate of $\rho:=\min\{1+\tau,\delta+q\}>1$. 
\label{theorem_sup_conv}
\end{theorem}
\begin{proof}
We define the constants
\[
\varepsilon_3 := \frac{1}{2+L_g}\left(\frac{\eta}{\hat c}\right)^{\frac{1}{\rho-1}}, \; \varepsilon_4 := \left(\frac{(1-c_1)\nu_{min}\beta}{cL(2+L_g)^{q-\delta}}\right)^{\frac{1}{q-\delta}}
\]
\[
\varepsilon_5 := \frac{1}{2+L_g}\left(\frac{1+L_g+\Vert A \Vert^2+\overline\nu}{\nu_{min}}\right)^{-\frac{1}{\kappa-1-\delta}}
\]
\[
\varepsilon_6 := \min\{\varepsilon_2,\varepsilon_3,\varepsilon_4,\varepsilon_5\}, \; \varepsilon_7 := \left(\varepsilon_6 / \left(1+\frac{c(2+L_g)^q}{\beta(1-\eta^q)}\right)\right)^{\frac{1}{\min(1,q)}}. 
\]
Assumption~\ref{ass_existence_cvx} ensures the existence of a subset $\mathcal L \subset \mathcal K$ with $\{x^k\}_{\mathcal L} \to x^*$. Consider some $k_0 \in \mathcal L$ with $x^{k_0} \in B_{\varepsilon_7}(x^*) \subset B_{\varepsilon_6}(x^*)$. We want to show that for all $k \geq k_0$, 
it holds that
\begin{subequations}
\label{eq_Lemma_l}
\begin{equation}
k \in \mathcal K, 
\label{eq_Lemma_L_1a}
\end{equation}
\begin{equation}
x^k \in B_{\varepsilon_6}(x^*).
\label{eq_Lemma_L_1b}
\end{equation}
\end{subequations}
For $k_0$ the above properties hold. Suppose now that \eqref{eq_Lemma_l} is satisfied for $k_0,...,k$ with some $k \geq k_0$. Using 
Lemma~\ref{lemma_ared_pred_ineq_cvx}, we then get
\begin{align*}
\text{ared}_k-c_3 \text{pred}_k &\geq \frac{1}{2}\left((1-c_3)\mu_k - \Vert \nabla^2f(\xi^k)-\nabla^2 f(x^k) \Vert\right)\Vert d^k \Vert^2 \\
&\geq \frac{1}{2}\left((1-c_3)\nu_{min}\Vert r(x^k) \Vert^{\delta-q} \Vert r(x^k) \Vert^q - cL \dist(x^k,X^*) \right)\Vert d^k \Vert^2 \\
&\geq \frac{1}{2}\left((1-c_3)\nu_{min}\beta\Vert r(x^k) \Vert^{\delta-q} - cL \right)\dist(x^k,X^*)\Vert d^k \Vert^2 \\
&\geq \frac{1}{2}\left((1-c_3)\nu_{min}\beta (2+L_g)^{\delta-q}\varepsilon_6^{\delta-q} - cL \right)\dist(x^k,X^*)\Vert d^k \Vert^2 \\
&\geq 0
\end{align*}
for some $c_3 \in (c_1,1)$, where the second inequality follows from 
Lemma~\ref{lemma_alg_prop_K}\ref{eq_lemma_alg_prop_K_eq}, \eqref{eq_locLip_hess_f} and Lemma~\ref{lemma_d^k_ineq_cvx}, the third from Assumption~\ref{ass_luo_tseng_errbound_cvx}, the fourth from 
Lemma~\ref{lemma_res_ineq} and \eqref{eq_Lemma_L_1b} together with $q > \delta$ and the fifth from the definition of $\varepsilon_6 \leq 
\varepsilon_4$. It follows that $\rho_k > c_1$. We also get
\begin{align*}
\frac{\Vert r(x^k) \Vert^{\kappa}}{\mu_k\Vert d^k \Vert} &= \Vert r(x^k) \Vert^{\kappa-1} \frac{\Vert r(x^k) \Vert}{\mu_k \Vert d^k \Vert} \leq \Vert r(x^k) \Vert^{\kappa-1} \frac{1+\Vert G_k \Vert}{(1-\theta)\mu_k} \leq \Vert r(x^k) \Vert^{\kappa-1-\delta} \frac{1+\Vert H_k \Vert+\mu_k}{(1-\theta)\nu_{min}} \\
&\leq (2+L_g)^{\kappa-1-\delta}\frac{1+L_g+\Vert A \Vert^2+\overline\nu}{(1-\theta)\nu_{min}} \varepsilon_6^{\kappa-1-\delta}\\
&\leq \frac{1}{1-\theta},
\end{align*}
where the first inequality follows from Lemma \ref{lemma_rk_dk}, the second from Lemma \ref{lemma_alg_prop_K}\ref{eq_lemma_alg_prop_K_eq} and $\nu_k \geq \nu_{min}$, the third from Lemma \ref{lemma_res_ineq}, \eqref{eq_Hess_loc_bd}, Lemma \ref{lemma_alg_prop_K}\ref{eq_lemma_alg_prop_K_mu}, $\Lambda_k \leq 1$ and $\Vert r(x^k) \Vert \leq 1$, and the fourth from the definition of $\varepsilon_6 \leq \varepsilon_5$. Together with \eqref{eq_pred} it then follows that
\[
\text{pred}_k \geq \frac{\mu_k}{2} \Vert d^k \Vert^2 \geq \frac{1-\theta}{2} \Vert d^k \Vert \Vert r(x^k) \Vert^\kappa > p_{min} (1-\theta) \Vert d^k \Vert \Vert r(x^k) \Vert^\kappa.
\]
Therefore, iteration $k$ is successful or highly successful. Furthermore it holds that 
\begin{align*}
\Vert r(x^{k+1}) \Vert &= \Vert r(\hat x^k) \Vert \leq \hat c\Vert r(x^k) \Vert^\rho = \hat c\Vert r(x^k) \Vert^{\rho-1} \Vert r(x^k) \Vert \\
&\leq \hat c \left( (2+L_g) \dist(x^k,X^*)\right)^{\rho-1} \Vert r(x^k) \Vert \\
&\leq \hat c (2+L_g)^{\rho-1}\left(\frac{1}{2+L_g}\left(\frac{\eta}{\hat c}\right)^{\frac{1}{\rho-1}}\right)^{\rho-1} \Vert r(x^k) \Vert = \eta \Vert r(x^k) \Vert = \eta\overline r_k,
\end{align*}
where we used \eqref{eq_lemma_superlinear_1} in the first, 
Lemma~\ref{lemma_res_ineq} and $\rho>1$ in the second, and the definition of $\varepsilon_6 \leq \varepsilon_3$ in the third inequality. In the last equality we used Lemma \ref{lemma_alg_prop_K}\ref{eq_lemma_alg_prop_K_eq}. It follows that $k+1 \in \mathcal K$. For all $j=k_0,...,k+1$ it holds that $j \in \mathcal K$ and thus
\begin{align}
\label{eq_lemma_L_2}
\Vert r(x^j) \Vert \leq \eta \overline r_{j-1} = \eta \Vert r(x^{j-1}) \Vert \leq ... \leq \eta^{j-k_0}\overline r_{k_0} = \eta^{j-k_0}\Vert r(x^{k_0}) \Vert,
\end{align}
by using Lemma~\ref{lemma_alg_prop_K}\ref{eq_lemma_alg_prop_K_dec} and 
Lemma~\ref{lemma_alg_prop_K}\ref{eq_lemma_alg_prop_K_eq} repeatedly. 
Moreover, it holds that $x^j = \hat x^{j-1} = x^{j-1}+d^{j-1}$ as all iterations $k_0,...,k$ are successful or highly successful by definition of $\mathcal K$. Thus we get
\begin{align}
\label{eq_theorem_sup_conv_1}
\begin{split}
\Vert x^{k+1}-x^{k_0} \Vert &= \sum_{j=k_0}^k \Vert d^j \Vert \leq c \sum_{j=k_0}^k \dist(x^{j},X^*) \leq \frac{c}{\beta} \sum_{j=k_0}^k \Vert r(x^j) \Vert^q \leq \frac{c}{\beta}\Vert r(x^{k_0}) \Vert^q \sum_{j=k_0}^k (\eta^q)^{j-k_0} \\
&\leq \frac{c}{\beta}\Vert r(x^{k_0}) \Vert^q \sum_{j=0}^\infty (\eta^q)^j = \frac{c}{\beta(1-\eta^q)}\Vert r(x^{k_0}) \Vert^q \leq \frac{c(2+L_g)^q}{\beta(1-\eta^q)}\Vert x^{k_0}-x^* \Vert^q,
\end{split}
\end{align}
where we used Lemma~\ref{lemma_d^k_ineq_cvx} in the first, 
Assumption~\ref{ass_luo_tseng_errbound_cvx} in the second, \eqref{eq_lemma_L_2} in the third and Lemma~\ref{lemma_res_ineq} in the last inequality. This implies
\begin{align*}
\Vert x^{k+1}-x^* \Vert &\leq \Vert x^{k+1}-x^{k_0} \Vert + \Vert x^{k_0}-x^* \Vert \leq \frac{c(2+L_g)^q}{\beta(1-\eta^q)}{\varepsilon_7}^q+\varepsilon_7 \\
&\leq \left(\frac{c(2+L_g)^q}{\beta(1-\eta^q)}+1\right)\varepsilon_7^{\min(1,q)}=\varepsilon_6.
\end{align*}
By induction, it follows that \eqref{eq_Lemma_l} holds for all $k \geq k_0$. For an iteration $k \geq k_0$ define $l_k$ as the iteration which satisfies the following three properties: 
\begin{equation}
l_k \leq k, \, l_k \in \mathcal L \text{ and } j \notin \mathcal L \text{ for } l_k<j\leq k.
\end{equation}
In words, $l_k$ is the last iteration belonging to $\mathcal L$ before iteration $k$. By construction it follows that $l_k \to \infty$ if $k \to \infty$ and therefore $\{x^{l_k}\} \to x^*$. Similar to \eqref{eq_theorem_sup_conv_1} it follows for $k \geq l_k \geq k_0$ that
\[
\Vert x^k-x^* \Vert \leq \Vert x^k-x^{l_k} \Vert + \Vert x^{l_k}-x^* \Vert \leq \frac{c(2+L_g)^q}{\beta(1-\eta^q)} \Vert x^{l_k}-x^* \Vert^q+\Vert x^{l_k}-x^* \Vert.
\]
Hence, $\{x^k\}$ converges to $x^*$. Now it immediately follows from \eqref{eq_lemma_superlinear_1} that $\{\Vert r(x^k) \Vert\}$ converges to $0$ at the rate of $\rho>1$.
\end{proof}

\begin{corollary}
Suppose that Assumption~\ref{assumption_loc_prox_cvx} holds with $q > \frac{1}{1+\delta}$. Then $\{x^k\}$ converges to $x^*$, $\{\Vert r(x^k) \Vert\}$ converges to $0$ at the rate of $\rho>1$ and $\{\dist(x^k,X^*)\}$ converges to $0$ at the rate of $(1+\delta)q>1$.
\end{corollary}

\begin{proof}
It holds that $\frac{1}{1+\delta} > \frac{1-\delta^2}{1+\delta} = \frac{(1+\delta)(1-\delta)}{1+\delta} = 1-\delta$, i.e. the assumption here is stronger than in Assumption~\ref{ass_luo_tseng_errbound_cvx}. The result follows directly from Theorem \ref{theorem_sup_conv} and \eqref{eq_lemma_superlinear_2}.
\end{proof}

\section{Numerical Results}
In this section, we present the numerical results of Algorithm \ref{algo_proxRegNewton_cvx} (denoted as IRPNM-reg) for various instances of Problem \ref{problem}. We compare these results with the outcomes of the inexact regularized proximal Newton method using line-search (IRPNM-ls) proposed in \cite{liu_pan_wu_yang_2024}, as well as a modern FISTA-type method (AC-FISTA) from \cite{liang_monteiro_2023}. \\
We start by considering the convex logistic regression problem with $l_1$-regularizer (Section 6.1) and group regularizer (Section 6.2). Subsequently, we investigate three non-convex problem classes introduced in \cite{liu_pan_wu_yang_2024}: $l_1$-regularized Student's $t$-regression (Section 6.3), Group regularized Student's $t$-regression (Section 6.4), and Restoration of a blurred image (Section 6.5). \\
For all tests, we fix the parameters for IRPNM-reg as follows: $c_1 = 10^{-4}$, $c_2 = 0.9$, $\sigma_1 = 0.5$, $\sigma_2 = 4$, $\eta = 0.9999$, $\theta = 0.9999$, $\alpha=0.99$, $a=1$, $\nu_{\min} = 10^{-8}$, $\nu_0 = \min \left(\frac{10^{-2}}{\max(1,\Vert r(x^0) \Vert)},10^{-4} \right)$, $\overline{\nu}  = 100$, $\delta = 0.45$, $\tau \geq \delta$, $p_{\min} = 10^{-8}$, and $\kappa = 2$. For IRPNM-ls and AC-FISTA, we adopt the recommended parameters from their respective papers. \\
The tests are conducted using Matlab R2022b on a 64-bit Linux system with an Intel(R) Core(TM) i5-3470 CPU @ 3.20GHz and 16 GB RAM. \\
Since IRPNM-reg solves exactly the same subproblems as IRPNM-ls, we employ the efficient strategy developed in \cite{liu_pan_wu_yang_2024}. This strategy solves the dual of an equivalent reformulation of \eqref{reg_proximal_newton_cvx} using an augmented Lagrangian method. The semismooth system of equations arising from the augmented Lagrangian method is solved using the semismooth Newton method. Notably, this strategy is tailored to address problems where $\psi$ is a separable function, a characteristic shared by many applications, including those under consideration here. For more comprehensive details on the subproblem solver, please refer to \cite[Section 5.1]{liu_pan_wu_yang_2024}. \\
We terminate each of the tested methods once the current iterate $x^k$ satisfies $\Vert r(x^k) \Vert \leq \texttt{tol}$. Here, $\texttt{tol}$ is chosen independently for each problem class and further distinguished between the two second order methods and AC-FISTA. 

\subsection{$l_1$-regularized Logistic Regression}
First we explore the logistic regression problem defined as
\begin{align}
\label{eq_log_reg_l1}
\min_{y, v} \frac{1}{m} \sum_{i=1}^m \log \left( 1 + \exp\left( -b_i(a_i^\top y + v)\right)\right) + \lambda \lVert y \rVert_1.
\end{align}
In this context, $a_i \in \mathbb{R}^n$ denotes feature vectors, $b_i \in \{-1,1\}$ represents corresponding labels for $i=1,...,m$, and we have $\lambda > 0$, $y \in \mathbb{R}^n$, and $v \in \mathbb{R}$. In standard instances of this problem, $m \gg n$. The logistic regression problem aligns with the general form of \eqref{problem}, where $\psi \colon \mathbb{R}^{n+1} \to \mathbb{R}$ is defined as

\[
\psi(u) := \frac{1}{m} \sum_{i=0}^m \log(1+\exp(-u_i)), \quad u := (y^\top, v)^\top,
\]

where the $i$-th row of the matrix $A \in \mathbb{R}^{m \times (n+1)}$ takes the form of $(b_ia_i^\top, b_i)$, and $b=0 \in \mathbb{R}^m$. The regularization function $\varphi \colon \mathbb{R}^{n+1} \to \mathbb{R}$ is given by $\varphi(u) := \lambda \lVert y \rVert_1$.

Following the methodology outlined in \cite{boyd_2011} and , we create test problems using $n=10^4$ feature vectors and $m=10^6$ training sets. Each $a_i$ has approximately $s \in \{10,100\}$ nonzero entries, independently sampled from a standard normal distribution. We choose $y^{\text{true}} \in \mathbb{R}^n$ with $10s$ non-zero entries and $v^{\text{true}} \in \mathbb{R}$, independently sampled from a standard normal distribution. Labels $b_i$ are determined by

\[
b_i = \text{sign}\left( a_i^\top y^{\text{true}} + v^{\text{true}} + v_i \right),
\]

where $v_i \in \mathbb{R}$, $i=1,...,m$, are generated independently from a normal distribution with variance $0.1$. Similar to \cite{kanzow_lechner_2021}, the regularization parameter $\lambda$ takes the form $c_{\lambda} \lambda_{\text{max}}$, with $c_{\lambda} \in \{1,0.1,0.01\}$, and

\[
\lambda_{\text{max}} = \frac{1}{m} \left\lVert \frac{m_-}{m} \sum_{i \colon b_i=1} a_i + \frac{m_+}{m} \sum_{i \colon b_i = -1} a_i \right\rVert,
\]

representing the smallest value such that $y^* = (0,v^*)$ is a solution of \eqref{eq_log_reg_l1}. Here, $m_+$ and $m_-$ represent the counts of indices where $b_i$ is equal to $+1$ or $-1$, respectively. The selection of this value is motivated in \cite{boyd_2007}. For each method, we select the starting point as $x^0=0$ and carry out 10 independent trials - that is, with ten sets of randomly generated data - for every combination of parameters $s$ and $c_{\lambda}$. Tables 1 and 2 present the averaged number of (outer) iterations, objective values, residuals and running times for the two second order methods and AC-FISTA, respectively. 
\begin{table}[H]
\centering
\begin{tabular}{c|c|cccc|cccc}
 &  & \multicolumn{4}{c|}{IRPNM-reg} & \multicolumn{4}{c}{IRPNM-ls} \\
\hline
$c_\lambda$ & s & iter & $F(x)$ & $\Vert r(x) \Vert$ & time & iter & $F(x)$ & $\Vert r(x) \Vert$ & time \\
\hline
1 & 10 & 63.0 & 0.0904 & 7.72e-06 & 39.8 & 63.0 & 0.0904 & 7.73e-06 & 34.8  \\
 & 100 & 4.4 & 0.4518 & 3.48e-06 & 18.1 & 4.4 & 0.4518 & 3.49e-06 & 17.8  \\
\hline
0.1 & 10 & 49.6 & 0.0785 & 9.98e-06 & 149.2 & 32.0 & 0.0785 & 9.99e-06 & 116.5 \\
 & 100 & 7.9 & 0.2434 & 9.62e-06 & 252.6 & 8.2 & 0.2434 & 9.63e-06 & 262.2 \\
\hline
0.01 & 10 & 117.3 & 0.0727 & 1.00e-05 & 227.4 & 87.3 & 0.0727 & 1.00e-05 & 193.2\\
 & 100 & 13.6 & 0.0844 & 9.92e-06 & 734.7 & 16.6 & 0.0844 & 9.98e-06 & 1038.2 
\end{tabular}
\caption{Averaged results of IRPNM-reg and IRPNM-ls for 10 independent trials with tolerance $\texttt{tol}=10^{-5}$.}
\end{table}
We observe that IRPNM-reg and IRPNM-ls produce identical objective values. Both methods exhibit improved performance for larger values of $c_\lambda$. Additionally, the algorithms perform better with sparser data ($s=10$) for $c_\lambda \in \{0.1,0.01\}$, but worse for $c_\lambda = 1$. The performance of the methods is comparable, with IRPNM-ls demonstrating slightly superior results in the case of $s=10$, while IRPNM-reg performs better when $s=100$ and $c_\lambda = 0.01$.
\begin{table}[H]
\centering
\begin{tabular}{c|c|cccc}
 &  & \multicolumn{4}{c}{AC-FISTA} \\
\hline
$c_\lambda$ & s & iter & $F(x)$ & $\Vert r(x) \Vert$ & time \\
\hline
1 & 10 & 21.2 & 0.0904 & 6.36e-06 & 13.4 \\
 & 100 & 10.0 & 0.4518 & 6.14e-06 & 50.9 \\
\hline
0.1 & 10 & 210.2 & 0.0784 & 9.73e-06 & 129.5 \\
 & 100 & 100.9 & 0.2434 & 8.02e-06 & 449.0 \\
\hline
0.01 & 10 & 272.4 & 0.0727 & 9.84e-06 & 162.9 \\
 & 100 & 179.8 & 0.0844 & 8.80e-06 & 759.7   
\end{tabular}
\caption{Averaged results of AC-FISTA for 10 independent trials with tolerance $\texttt{tol}=10^{-5}$.}
\end{table}
AC-FISTA produces nearly identical objective values as the second-order methods. It generally outperforms the second-order methods for $s=10$ but performs worse for $s=100$, with some exceptions. Notably, in the case of $s=10$ and $c_\lambda=0.1$, IRPNM-ls is slightly faster than AC-FISTA. Conversely, for $s=100$ and $c_\lambda=0.01$, AC-FISTA significantly outperforms IRPNM-ls, nearly matching the runtime of IRPNM-reg.
\subsection{Group regularized Logistic Regression}
We consider the group regularized logistic regression problem, given by 
\begin{align*}
\min_{y, v} \frac{1}{m} \sum_{i=1}^m \log \left( 1 + \exp\left( -b_i(a_i^\top y + v)\right)\right) + \lambda \sum_{i=1}^l \Vert x_{J_i} \Vert_2,
\end{align*}
where the data $a_i \in \R^n$, $b_i \in \{-1,1\}$ for $i=1,...,m$ and $v \in \R$ follows the same generation process as in section 6.2 (with $s=10$). The index sets $J_1,...,J_l$ form a partition of $\{1,...,n\}$, i.e. they satisfy $J_i \cap J_j = \emptyset$ for $i \neq j$ and $\cup_{i=1}^l J_i = \{1,...,n\}$. We organize the $n=10^4$ in two different configurations: $l=1000$ groups of $100$ variables and $l=100$ groups of $1000$ variables, while consistently preserving a sequential group structure. The regularization parameter $\lambda$ mirrors the one in section 6.2 with $c_\lambda \in \{1,0.1,0.01\}$, and the initial value is set as $x^0 = 0$. Similar to the previous test problem, we conduct 10 independent trials for each value of $c_\lambda$. Tables 3 and 4 present the averaged number of (outer) iterations, objective values, residuals and running times for the two second order methods and AC-FISTA, respectively. 
\begin{table}[H]
\centering
\begin{tabular}{c|c|cccc|cccc}
 & & \multicolumn{4}{c|}{IRPNM-reg} & \multicolumn{4}{c}{IRPNM-ls} \\
\hline
$l$ & $c_\lambda$ & iter & $F(x)$ & $\Vert r(x) \Vert$ & time & iter & $F(x)$ & $\Vert r(x) \Vert$ & time \\
\hline
1000 & 1 & 7.3 & 0.3049 & 8.22e-06 & 14.7 & 7.4 & 0.3049 & 8.57e-06 & 15.3 \\
& 0.1 & 11.6 & 0.2725 & 9.93e-06 & 98.8 & 9.6 & 0.2725 & 9.96e-06 & 79.1 \\
& 0.01 & 23.2 & 0.2574 & 9.98e-06 & 184.8 & 22.2 & 0.2574 & 1.00e-05 & 177.0  \\
\hline 
100 & 1 & 13.8 & 0.3039 & 9.74e-06 & 63.7 & 7.1 & 0.3039 & 9.80e-06 & 36.8 \\
& 0.1 & 10.0 & 0.2690 & 9.92e-06 & 98.6 & 9.9 & 0.2690 & 9.98e-06 & 92.1 \\
& 0.01 & 25.0 & 0.2560 & 9.99e-06 & 255.9 & 24.3 & 0.2560 & 1.00e-05 & 194.7  
\end{tabular}
\caption{Averaged results of IRPNM-reg and IRPNM-ls for 10 independent trials with tolerance $\texttt{tol}=10^{-5}$.}
\end{table}
Both methods yield the same objective values in essentially the same run times. IRPNM-ls is slightly faster than IRPNM-reg across all test instances.
\begin{table}[H]
\centering
\begin{tabular}{c|c|cccc}
 & & \multicolumn{4}{c}{AC-FISTA} \\
\hline
$l$ & $c_\lambda$ & iter & $F(x)$ & $\Vert r(x) \Vert$ & time \\
\hline
1000 & 1 & 44.3 & 0.3049 & 8.23e-05 & 28.1 \\
& 0.1 & 134.3 & 0.2726 & 9.23e-05 & 81.4 \\
& 0.01 & 209.8 & 0.2582 & 9.84e-05 & 122.2 \\ 
\hline   
100 & 1 & 151.5 & 0.3039 & 9.44e-06 & 92.6 \\
& 0.1 & 307.9 & 0.2690 & 9.76e-06 & 251.8 \\
& 0.01 & 607.6 & 0.2560 & 9.90e-06 & 343.9 \\ 
\end{tabular}
\caption{Averaged results of AC-FISTA for 10 independent trials with tolerance $\texttt{tol}=10^{-5}$.}
\end{table}
AC-FISTA achieves the same objective values as the second-order methods. When $l=100$, AC-FISTA underperforms compared to the second-order methods. For $l=1000$, AC-FISTA exhibits inferior performance for large $c_\lambda$ values but superior performance for smaller $c_\lambda$ values.
\subsection{$l_1$-regularized Student's $t$-regression}
We consider the Student's $t$-regression problem with $l_1$-regularizer, given by
\begin{align*}
\min_{x} \sum_{i=1}^m \log(1+(Ax-b)_i/\nu) + \lambda \Vert x \Vert_1,
\end{align*}
where $A \in \R^{m \times n}$, $b \in \R^m$, $\nu > 0$ and $\lambda > 0$. The test examples are randomly generated following the same procedure as in \cite{becker_bobin_candes_2011, liu_pan_wu_yang_2024, milzarek_ulbrich_2014}. The matrix $A$ is formed by taking $m = n/8$ random cosine measurements, i.e. $Ax = (\texttt{dct}(x))_J$, where $\texttt{dct}$ denotes the discrete cosine transform, and $J \subseteq \{1,...,n\}$ is an index set selected at random with $\vert J \vert = m$. A true sparse signal $x^{true}$ of length $n=512^2$ is created, featuring $s = \lfloor \frac{n}{40} \rfloor$ randomly selected non-zero entries, calculated as $x_i^{\text{true}} = \eta_1(i)10^{\frac{d \eta_2(i)}{20}}$, where $\eta_1(i) \in \{-1,1\}$ denotes a random sign and $\eta_2(i)$ is uniformly distributed in the interval $[0,1]$. The signal possesses a dynamic range of $d$ dB with $d \in \{20,40,60,80\}$. The vector $b$ is then obtained by summing $Ax^{\text{true}}$ and Student's $t$-noise with a degree of freedom of $4$, rescaled by $0.1$. \\
The regularization parameter is expressed as $\lambda = c_\lambda \Vert \nabla f(0) \Vert_\infty$, where $c_\lambda \in \{0.1,0.01\}$. For each combination of values $d$ and $c_\lambda$ we run the three solvers with $\nu = 0.25$ and $x^{init} = A^\top b$ over $10$ independent trials. Tables 5 and 6 present the averaged number of (outer) iterations, objective values, residuals and running times for IRPNM-reg and IRPNM-ls with $\texttt{tol} = 10^{-5}$, and AC-FISTA with $\texttt{tol} = 10^{-4}$, respectively.
\begin{table}[H]
\centering
\begin{tabular}{c|c|cccc|cccc}
&  & \multicolumn{4}{c|}{IRPNM-reg} & \multicolumn{4}{c}{IRPNM-ls} \\
\hline
$c_\lambda$ & $d$ & iter & $F(x)$ & $\Vert r(x) \Vert$ & time & iter & $F(x)$ & $\Vert r(x) \Vert$ & time \\
\hline
0.1 & 20 & 28.4 & 9532.5413 & 8.78e-06 & 13.5 & 24.2 & 9532.5413 & 8.92e-06 & 13.1 \\
 & 40 & 19.5 & 23812.8786 & 6.00e-06 & 32.0 & 17.2 & 23812.8749 & 6.75e-06 & 33.3 \\
 & 60 & 24.7 & 54228.0069 & 8.07e-06 & 88.1 & 23.8 & 54228.0069 & 6.85e-06 & 84.2 \\
 & 80 & 80.3 & 134779.2596 & 8.54e-06 & 281.5 & 109.7 & 134779.2596 & 8.03e-06 & 323.2 \\
\hline
0.01 & 20 & 11.8 & 1020.4271 & 7.09e-06 & 37.2 & 8.9 & 1020.4271 & 7.08e-06 & 37.0 \\
 & 40 & 15.5 & 2395.0693 & 7.90e-06 & 129.4 & 14.1 & 2395.0693 & 7.63e-06 & 122.4 \\
 & 60 & 12.4 & 5424.4039 & 7.33e-06 & 170.5 & 17.7 & 5424.4039 & 7.65e-06 & 261.9 \\
 & 80 & 16.3 & 13478.1029 & 6.17e-06 & 314.2 & 116.3 & 13478.1029 & 7.50e-06 & 1103.9 \\  
\end{tabular}
\caption{Averaged results of IRPNM-reg and IRPNM-ls for 10 independent trials with tolerance $\texttt{tol}=10^{-5}$.}
\end{table}
Both methods yield the same objective values except for the case $c_\lambda = 0.1$, $d = 40$, where IRPNM-ls yields (in average) a slightly smaller objective value. In most cases, the runtimes for the two methods are comparable. However, for $c_\lambda = 0.01$ and $d \in \{60,80\}$ IRPNM-reg performs better, requiring only a third of the runtime of IRPNM-ls for $d=80$. 
\begin{table}[H]
\centering
\begin{tabular}{c|c|cccc}
&  & \multicolumn{4}{c}{AC-FISTA} \\
\hline
$c_\lambda$ & $d$ & iter & $F(x)$ & $\Vert r(x) \Vert$ & time \\
\hline
0.1 & 20 & 507.5 & 9532.5413 & 9.71e-05 & 31.4 \\
 & 40 & 1041.5 & 23812.8749 & 9.84e-05 & 95.2 \\
 & 60 & 2238.9 & 54228.0069 & 9.90e-05 & 134.6 \\
 & 80 & 7240.7.5 & 134779.2596 & 9.95e-05 & 434.2 \\
\hline
0.01 & 20 & 1488.8 & 1020.4271 & 9.97e-05 & 98.0 \\
 & 40 & 2531.7 & 2395.0693 & 9.98e-05 & 162.8 \\
 & 60 & 5391.4 & 5424.4039 & 9.94e-05 & 327.9 \\
 & 80 & 20694.0 & 13478.1029 & 9.93e-05 & 1243.0 \\  
\end{tabular}
\caption{Averaged results of AC-FISTA for 10 independent trials with tolerance $\texttt{tol}=10^{-4}$.}
\end{table}
Note that here we chose $\texttt{tol} = 10^{-4}$ for AC-FISTA instead of $10^{-5}$. It solves all the problems and returns the same objective values. In all cases it takes longer to solve the problems with tolerance $10^{-5}$ than both second order methods with tolerance $10^{-6}$. However, in some cases (e.g. $c_\lambda=0.01$ and $d=80$), it does not perform much worse than IRPNM-ls.

\subsection{Group penalized Student's $t$-regression}
We consider the Student's $t$-regression problem with group regularizer, given by
\begin{align*}
\min_{x} \sum_{i=1}^m \log(1+(Ax-b)_i/\nu) + \lambda \sum_{i=1}^l \Vert x_{J_i} \Vert_2.
\end{align*}
This test problem is taken from \cite[Section 5.3]{liu_pan_wu_yang_2024}. A true group sparse signal $x^{true} \in \R^n$ of length $n = 512^2$ with $s$ nonzero groups is generated, whose indices are chosen randomly. Each nonzero entry of $x^{true}$ is calculated using the same formula as in section 6.3. The matrix $A \in \R^{m \times n}$ and the vector $b \in \R^m$ are also obtained in the same way as in section 6.3, with the only difference being the choice of degree of freedom $5$ for the Student's $t$-noise. \\
The regularization parameter is set as $\lambda = 0.1 \Vert \nabla f(0) \Vert$. For each combination of values $d \in \{60,80\}$ dB and non-zero groups $s=\{16,64,128\}$ we run the three solvers with $\nu = 0.2$ and $x^{init} = A^\top b$ over $10$ independent trials. Tables 7 and 8 present the averaged number of (outer) iterations, objective values, residuals and running times for IRPNM-reg and IRPNM-ls with $\texttt{tol} = 10^{-5}$, and AC-FISTA with $\texttt{tol} = 10^{-3}$, respectively.
\begin{table}[H]
\centering
\begin{tabular}{c|c|cccc|cccc}
&  & \multicolumn{4}{c|}{IRPNM-reg} & \multicolumn{4}{c}{IRPNM-ls} \\
\hline
$d$ & $s$ & iter & $F(x)$ & $\Vert r(x) \Vert$ & time & iter & $F(x)$ & $\Vert r(x) \Vert$ & time \\
\hline
 60 & 16 & 6.1 & 12711.8673 & 6.54e-06 & 16.79 & 9.0 & 12711.8673 & 8.44e-06 & 16.69 \\
  & 64 & 6.7 & 17852.9902 & 8.65e-06 & 19.53 & 12.0 & 17852.9902 & 8.08e-06 & 26.06 \\
  & 128 & 7.0 & 21670.1861 & 8.64e-06 & 20.16 & 14.9 & 21670.1861 & 9.13e-06 & 34.48 \\
\hline
 80 & 16 & 9.0 & 37037.7136 & 9.26e-06 & 38.30 & 54.8 & 37037.7137 & 9.67e-06 & 133.25 \\
  & 64 & 11.0 & 52741.5880 & 7.40e-06 & 49.86 & 91.7 & 52741.5881 & 9.77e-06 & 245.62 \\
  & 128 & 13.3 & 63451.7421 & 7.07e-06 & 61.90 & 128.2 & 63451.7421 & 9.40e-06 & 372.52 \\ 
\end{tabular}
\caption{Averaged results of IRPNM-reg and IRPNM-ls for 10 independent trials with tolerance $\texttt{tol}=10^{-5}$.}
\end{table}
Both methods produce - essentially - the same objective values. IRPNM-reg shows better performance than IRPNM-ls for $d=60$ and significantly better for $d=80$. \\
\begin{table}[H]
\centering
\begin{tabular}{c|c|cccc}
&  & \multicolumn{4}{c}{AC-FISTA} \\
\hline
$d$ & $s$ & iter & $F(x)$ & $\Vert r(x) \Vert$ & time \\
\hline
60 & 16 & 4204.3 & 12711.8731 & 9.91e-04 & 318.76 \\
 & 64 & 6282.8 & 17853.0157 & 1.00e-03 & 438.99 \\
 & 128 & 8936.6 & 21670.2322 & 1.00e-03 & 592.16 \\

\hline
80 & 16 & 20954.6 & 37037.9708 & 9.97e-04 & 1493.72 \\
 & 64 & 30273.6 & 52742.3811 & 9.93e-04 & 1926.70 \\
 & 128 & 31849.3 & 63452.8920 & 9.91e-04 & 1925.60 \\
\end{tabular}
\caption{Averaged results of AC-FISTA for 10 independent trials with tolerance $\texttt{tol}=10^{-3}$.}
\end{table}
In this example, we had to select $\texttt{tol} = 10^{-3}$ for AC-FISTA. It is evident that this reduced accuracy leads to higher average objective values. For this problem class, AC-FISTA is clearly outperformed by both second order methods.

\subsection{Nonconvex Image Restoration}

In this section we apply the algorithms to image restoration using real-world data. The problem is the same as in \cite{lechner_2022,liu_pan_wu_yang_2024}. The goal is to find an approximation $x \in \R^n$ of the original image $x^{true} \in \R^n$ from a noisy blurred image $b \in \R^n$ and a blur operator $A \in \R^{n \times n}$, i.e., we seek $x$ with $Ax \approx b$. The objective function incorporates a regularization term $\lambda \Vert Bx \Vert_1$ to ensure smooth gradations and antialiasing in the final image, where $B \colon \R^n \to \R^n$ is a two-dimensional discrete Haar wavelet transform. The problem can be expressed as
\begin{align*}
\min_{x} \sum_{i=1}^m \log(1+(Ax-b)_i) + \lambda \Vert Bx \Vert_1,
\end{align*}
with $\lambda > 0$. Making use of the orthogonality of $B$, the problem can be reformulated equivalently as
\begin{align*}
\min_{y} \sum_{i=1}^m \log(1+(AB^\top y-b)_i) + \lambda \Vert y \Vert_1,
\end{align*}
which clearly is an instance of the problem class considered in section 6.3. 

The test setup being identical to \cite{lechner_2022,liu_pan_wu_yang_2024}, we select the $256 \times 256$ grayscale image \texttt{cameraman.tif} as the test image $x^{true} \in \R^n$ with $n = 256^2$. The blur operator $A$ is a $9 \times 9$ Gaussian filter with a standard deviation of 4, and $B$ is a two-dimensional discrete Haar wavelet of level 4. The noisy image $b$ is created by applying $A$ to the original cameraman test image $x^{true}$ and adding Student’s t-noise with degree of freedom 1 and rescaled by $10^{-3}$. For each $\lambda \in \{10^{-2},10^{-3},10^{-4}\}$, we run the three solvers with $y^{init} = Bb$ and $\texttt{tol}=10^{-5}$ for $10$ independent trials. Here we decided to use $\nu_{\text{min}} = 10^{-4}$ instead of $10^{-8}$. The reason for this change is that in this test scenario, instances where the subproblem couldn't be solved within the desired maximum number of iterations were much more frequent. Consequently, a significantly higher number of unsuccessful iterations occurred. It is noteworthy that these unsuccessful iterations tend to negatively affect IRPNM-reg more than IRPNM-ls. This is because the line search enables the algorithm to still make some progress, whereas IRPNM-reg simply repeats solving the same subproblem with a larger regularization parameter. Given that subproblems become more challenging to solve with smaller regularization parameters, selecting $\nu_{\text{min}} = 10^{-4}$ instead of $10^{-8}$ notably reduced the number of unsuccessful iterations and consequently enhanced the performance of IRPNM-reg. Additionally, for this particular example, we experimented with a hybrid approach, IRPNM-reg-ls, which combines both methods. In IRPNM-reg-ls, a line search is conducted whenever an unsuccessful iteration occurs. Table 9 presents the averaged number of (outer) iterations, objective values, residuals and running times for the three second order methods and AC-FISTA. 
\begin{table}[H]
\centering
\begin{tabular}{c|cccc|cccc}
& \multicolumn{4}{c|}{IRPNM-reg} & \multicolumn{4}{c}{IRPNM-ls} \\
\hline
$\lambda$ & iter & $F(x)$ & $\Vert r(x) \Vert$ & time & iter & $F(x)$ & $\Vert r(x) \Vert$ & time \\
\hline
 1e-2 & 97.4 & 11245.2731 & 9.84e-05 & 200.15 & 99.3 & 11245.2731 & 9.77e-05 & 201.37 \\
 1e-3 & 115.1 & 1199.4475 & 9.38e-05 & 465.71 & 113.1 & 1199.4475 & 8.77e-05 & 427.08 \\
 1e-4 & 122.7 & 146.9925 & 9.10e-05 & 709.44 & 121.3 & 146.9927 & 9.55e-05 & 667.09 \\
\hline
& \multicolumn{4}{c|}{IRPNM-reg-ls} & \multicolumn{4}{c}{AC-FISTA} \\
\hline
1e-2 & 97.4 & 11245.2731 & 9.84e-05 & 199.33 & 2086.7 & 11245.2731 & 9.88e-05 & 279.63 \\
1e-3 & 113.3 & 1199.4475 & 9.66e-05 & 445.34 & 3486.9 & 1199.3795 & 9.86e-05 & 494.03 \\
1e-4 & 118.7 & 146.9926 & 9.20e-05 & 647.46 & 6825.9 & 146.7919 & 9.85e-05 & 908.78
\end{tabular}
\caption{Averaged results of IRPNM-reg, IRPNM-ls, IRPNM-reg-ls and AC-FISTA for 10 independent trials with tolerance $\texttt{tol}=10^{-4}$.}
\end{table}
All three second order methods produce similar objective values for all instances, with IRPNM-ls showing slightly better performance than IRPNM-reg across all different choices of the regularization parameter $\lambda$. The hybrid method IRPNM-reg-ls yields similar results as IRPNM-ls, performing slighty worse for $\lambda=10^{-3}$ and slightly better for $\lambda = 10^{-4}$. We can see that AC-FISTA converges (on average) to slightly better stationary points than the second order methods for $\lambda = 10^{-3}$ and $\lambda = 10^{-4}$. Additionally, it demonstrates good runtime performance.
\begin{figure}[H]
\centering
\begin{subfigure}[b]{0.25\textwidth}
\includegraphics[width=\textwidth]{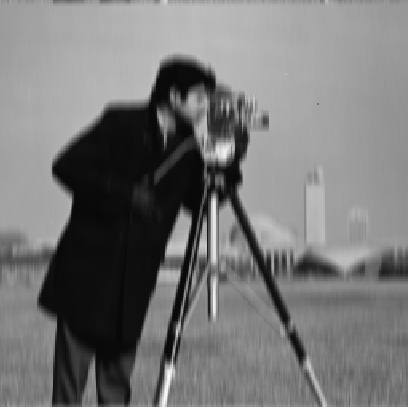}
\caption{Noisy blurred image}
\end{subfigure}
\begin{subfigure}[b]{0.25\textwidth}
\includegraphics[width=\textwidth]{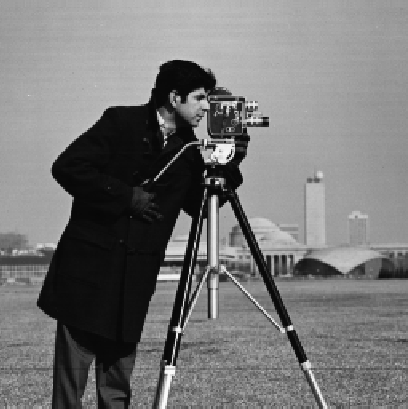}
\caption{Original image}
\end{subfigure}
\begin{subfigure}[b]{0.25\textwidth}
\includegraphics[width=\textwidth]{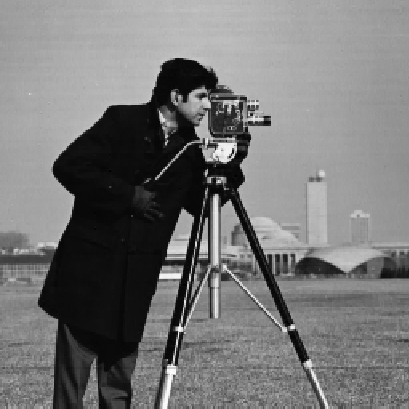}
\caption{IRPNM-reg}
\end{subfigure}
\caption{Nonconvex image restoration with IRPNM-reg for $\lambda = 10^{-2}$ and $\texttt{tol}=10^{-4}$ (reconstructed images with IRPNM-ls and AC-FISTA are omitted since they are indistinguishable from those obtained with IRPNM-reg).}
\end{figure}
\section{Final Remarks}
In this work, we introduced an inexact proximal Newton method without line search, ensuring global convergence through a careful update strategy for the regularization parameter based on the previous iteration. A superlinear convergence rate of the iterate sequence was shown under a local Hölderian error bound condition and confirmed in numerical tests across various problem classes. \
Our findings suggest several avenues for future research. Similar convergence results, i.e. without requiring a global Lipschitz assumption on $\nabla f$, may be achievable for an inexact proximal Newton method using line search. Exploring analogous outcomes for a proximal Quasi-Newton method is another potential research direction. Additionally, a convergence analysis for $\delta=0$ could be pursued under the assumption that $F$ is a KL (Kurdyka-\L ojawiewicz) function, following the approach in \cite{liu_pan_wu_yang_2024}. \\

\setlength{\parindent}{0cm}

\small{\textbf{Acknowledgements} \, The authors express their gratitude to Jiaming Liang from Yale University, Renato D. C. Monteiro from the Georgia Institute of Technology, Ruyu Liu and Shaohua Pan from the South China University of Technology, as well as Yuqia Wu and Xiaoqi Yang from the Hong Kong Polytechnic University for generously sharing their code.} \\

\small{\textbf{Data availability} \, The test problems in Section 6 are based on randomly generated data, except for the restoration of a blurred image. The Matlab code employed for both data generation and numerical tests is accessible upon request from Simeon vom Dahl (simeon.vomdahl@uni-wuerzburg.de)} 

\bibliography{RegProxNewton_PP}
\end{document}